\documentclass[11pt,a4paper]{amsart}
\usepackage[utf8]{inputenc}
\usepackage{ucs}
\usepackage{amsmath}
\usepackage{amsthm}
\usepackage{amsfonts}
\usepackage{amssymb,mathrsfs}
\usepackage[english]{babel}
\usepackage[cyr]{aeguill}
\usepackage{enumerate}
\usepackage{geometry}
\usepackage{blkarray}
\usepackage[bookmarks=true]{hyperref}
\hypersetup{colorlinks=true,linkcolor=red,citecolor=blue}
\usepackage{pgf,tikz}
\usetikzlibrary{arrows}
\newtheorem{theo}{Theorem}[section]
\newtheorem*{theo*}{Theorem}
\newtheorem{cor}[theo]{Corollary}
\newtheorem{lem}[theo]{Lemma}
\newtheorem{prop}[theo]{Proposition}
\theoremstyle{remark}
\newtheorem{rem}[theo]{Remark}

\newtheorem*{ack}{Acknowledgments}

\theoremstyle{definition}
\newtheorem{df}[theo]{Definition}
\newtheorem*{df*}{Definition}

\newcommand{\R}{\mathbb{R}}

\newcommand{\N}{\mathbb{N}}
\newcommand{\C}{\mathbb{C}}
\renewcommand{\H}{\mathbb{H}}
\newcommand{\K}{\mathbb{K}}

\renewcommand{\O}{\mathbf{O}}
\newcommand{\E}{\mathbf{E}}

\newcommand{\X}{\mathbf{X}}

\newcommand{\cangle}{\overline{\angle}}
\newcommand{\bo}{\partial}

\renewcommand\cosh{\mathrm{ch}}
\renewcommand\sinh{\mathrm{sh}}
\numberwithin{equation}{section}
\renewcommand{\epsilon}{\varepsilon}

\DeclareMathOperator{\CAT}{CAT}
\DeclareMathOperator{\diam}{diam}

\DeclareMathOperator{\Isom}{Isom}

\DeclareMathOperator{\Id}{Id}
\DeclareMathOperator{\diag}{diag}
\newcommand{\Hil}{\mathcal{H}}
\newcommand{\action}{\curvearrowright}

\begin{document}
\title{Infinite dimensional non-positively curved\\ symmetric spaces of finite rank}
\author{Bruno Duchesne}
\address{Section de Mathématiques de l'Université de Genève $\&$ UMPA, ENS de Lyon, Université de Lyon.}
\email{bruno.duchesne@unige.ch}
\thanks{Partially supported by  the French ANR project  AGORA : NT09\_461407.}
\date{\today}
\begin{abstract}This paper concerns a study of three families of non-compact type symmetric spaces of infinite dimension. Although they have infinite dimension they have finite rank. More precisely, we show they have finite telescopic dimension. We also show the existence of Furstenberg maps for some group actions on these spaces. Such maps appear as a first step toward superrigidity results.
\end{abstract}
\maketitle
\tableofcontents
\section{Introduction}
Riemannian symmetric  spaces of non-compact type have been introduced and classified by E. Cartan in the 1920s. Since then, they have always been closely related to semi-simple Lie groups; indeed, there is a dictionary between semi-simple Lie groups with finite center  and without compact factor and Riemannian symmetric spaces of non-compact type (the ones with non-positive  sectional curvature and without Euclidean de Rham factor). We refer to  \cite{MR1834454} for a general theory of Riemannian symmetric spaces. Euclidean buildings play a similar role for semi-simple algebraic groups over non-archimedean local fields. \\

Some great results of rigidity where obtained by G. Mostow and G. Margulis in the 1970s. We are inspired by the following way to state Margulis superrigidity theorem.
\begin{theo}[Théorème 2 in \cite{MR2655318}] Let $X,Y$ be two Riemannian symmetric spaces of non-compact type  or Euclidean buildings and $\Gamma$ a lattice in Isom$(X)$. Assume that  $X$ is irreducible and its rank is larger than 1.\\
If $\Gamma$ acts non-elementarily  by isometries on $Y$ then $\Gamma$  preserves a closed invariant subspace $Z\subseteq Y$, which   is isometric to $X$, and the action of $\Gamma$ on  $Z$ extends to Isom$(X)$.
\end{theo}

We aim at a similar statement in an infinite dimensional  setting. However things are not straightforward, as the following remarks indicate. 
\begin{enumerate}[(i)]
\item Isometries of infinite dimensional Hilbert spaces can be very wild (see \cite{MR0164222}). 
\item A rather natural idea is to consider $X=\mathbf{GL}(\mathcal{H})/\O(\mathcal{H})$ where $\mathbf{GL}(\mathcal{H})$ is the group of all invertible bounded operators of a Hilbert space $\mathcal{H}$ and $\O(\mathcal{H})$ is its orthogonal group; the one hand, $X$ is some generalisation of SL$_n(\R)/$SO$_n(\R)$, in which any Riemannian symmetric space of non-compact type embeds as a totally geodesic submanifold (for n large enough); but, on the other hand, $X$ is not a manifold modelled on a Hilbert space, and not a $\CAT(0)$ space. In particular, a group that acts by isometries on $X$ and preserves a bounded subset does not have necessarily a fixed point. Actually, it is a Banach-Finsler manifold of non-positive curvature in the sense of Busemann (see \cite{MR1950888}).
\item A better candidate for an infinite-dimensional analogue of SL$_n(\R)/$SO$_n(\R)$ could be $\mathbf{GL}^2(\mathcal{H})/\O^2(\mathcal{H})$, where  $\mathbf{GL}^2(\mathcal{H})$ denotes the subgroup of invertible operators $G$ such that $G-I$ is a Hilbert-Schmidt operator and $\O^2(\mathcal{H})=\mathbf{GL}^2(\mathcal{H})\cap\O(\mathcal{H})$. This is indeed an infinite-dimensional Riemannnian Hilbert manifold and a $\CAT(0)$ space, but it has infinite rank (as defined below).
\end{enumerate}

In \cite[section 6]{MR1253544}, M.Gromov suggests the study of the Riemannian symmetric spaces $X_p=\O(p,\infty)/\O(p)\times\O(\infty)$.  These spaces seem to him as ``\emph{cute and sexy as their finite dimensional siblings}". Moreover he suggests a similar statement to Margulis' superrigidity  should be true for actions of semi-simple Lie groups on these spaces.
\subsection{Geometry}\label{iGeo}
Let $\R,\C,\H$ be respectively the fields of real numbers, complex numbers and quaternions of W.R. Hamilton. Throughout this article $\K$ will denote one of these three previous fields and $\lambda\mapsto \overline{\lambda}$ will denote the conjugation in $\K$ (which is the identity map if $\K=\R$). Let $\mathcal{H}$ be a (right) $\K$-vector space. A sesquilinear form on $\mathcal{H}$ is a map $x,y\mapsto<x,y>$ such that $<x,y>=\overline{<y,x>}$ and $<x, y\lambda+z>=<x,y>\lambda+<x,z>$ for all $x,y,z\in\mathcal{H}$ and $\lambda\in\K$. The sesquilinear form $<,>$ is said to be positive definite if $<x,x>\geq0$ for all $x\in \mathcal{H}$ and $\left(<x,x>=0\right)\Rightarrow \left(x=0\right)$.\\

\begin{rem} Our choice of \emph{right} vector spaces allows us to identify linear maps with matrices with coefficients in $\K$ in such a way that matrices act by left multiplication on vectors; moreover, compositions of maps correspond to usual multiplication of matrices. Our choice is consistent with the choice made in \cite[Chapter II.10]{MR1744486} but different of the one in  \cite[Section 0.1]{MR0150210} for example.
\end{rem}

A $\K$-vector space has a structure of $\R$-vector by restriction of the scalars; we will denote by $\Hil_\R$ this structure. If $<,>$ is a sesquilinear form on $\Hil$, we define $<x,y>_\R=$Re$(<x,y>)$, which is a symmetric linear form on the real space $\Hil_\R$. We say that $(\Hil,<,>)$ is a $\K$-Hilbert space if $(\Hil_\R, <,>_\R)$ is a real Hilbert space. In this case, the norm of an element $x\in\mathcal{H}$ is $||x||=\sqrt{<x,x>}=\sqrt{<x,x>_\R}$. The topology on $\mathcal{H}$ is defined by the metric associated with the norm $||\ ||$ and does not depend on the field $\R$ or $\K$.\\

Let $p$ be a positive integer, $E_0$ a linear subspace of dimension $p$ and $\Phi$ the linear operator of $\mathcal{H}$ such that $\Phi|_{E_0}=$Id and $\Phi|_{E_0^\bot}=-$Id. We define a new sesquilinear form by
\[B_p(x,y)=<x,\Phi(y)>\  \mathrm{and}\ \mathrm{set}\ Q_p(x)=B_p(x,x),\  \mathrm{for}\ x,y\in\Hil.\]

Suppose $\Hil$ is infinite dimensional and separable. Let $\mathcal{G}_p$ be the Grasmannian of all $\K$-linear subspaces of dimension $p$ in $\Hil$ then we define 
\[X_p(\K)=\left\{E\in\mathcal{G}_p\left|\ B_p|_{E\times E}\ \mathrm{is\ positive\ definite}\right.\right\}.\] 

We show $X_p(\K)$ is a Riemannian symmetric manifold of infinite dimension. Moreover it has non-positive sectional curvature and thus is a complete CAT(0) space. We show that at ``large scale" $X_p(\K)$ behaves really like a finite dimensional Riemannian symmetric space of non-compact type. In particular its boundary $\bo X_p(\K)$ is a spherical building of dimension $p-1$ (see section \ref{immeubles}).

\begin{theo}\label{eb} Every asymptotic cone over $X_p(\K)$ is a Euclidean building of rank $p$ in the sense of Kleiner and Leeb.
\end{theo}
Following \cite{MR2558883}, we say that a $\CAT(0)$ space, $X$, has finite telescopic dimension if there exists $n\in\N$ such that every asymptotic cone of $X$ has geometric dimension at most $n$. In this case, the \emph{telescopic dimension} of $X$ is the minimum among such $n$. The \emph{rank} of a $\CAT(0)$ space, $X$, is the maximal dimension of an Euclidean space isometrically embedded in $X$. The rank of $X$ is not greater than the telescopic dimension of $X$. For Riemannian symmetric spaces of non-compact type, rank and telescopic dimension coincide. By Theorem \ref{eb}, this is also true for $X_p(\K)$ : 
\begin{cor}\label{td}The space $X_p(\K)$ is a separable complete CAT(0) space of rank and telescopic dimension $p$.
\end{cor}
Thus, all nice properties of finite telescopic dimension spaces hold for $X_p(\K)$. For  example, every parabolic isometry of $X_p(\K)$ has a canonical (but not necessarily unique) fixed point in $\bo X_p(\K)$ \cite[Corollary 1.5]{MR2558883}. Any continuous action of an amenable group on $X_p(\K)$ satisfies a comparable conclusion to a result of S. Adams and W. Ballmann \cite{MR1645958}; such action  has a fixed point at infinity or stabilizes a finite dimensional Euclidean subspace in $X_p(\K)$  \cite[Theorem 1.6]{MR2558883}.\\

We give a more concrete expression of the metric. We introduce \emph{hyperbolic principal angles}. Recall first that, for two linear subspaces of dimension $p$ in a Euclidean space $\R^n$, there is a well-known construction (due to C. Jordan in \cite{MR1503705}) of a family of angles $(\theta_1,\dots,\theta_p)$ called \emph{principal angles} between them, which generalize the angle between two lines. If $(E,F)$ and $(E',F')$ are pairs of linear subspaces of dimension $p$ then there is $g\in\O(n)$ such that $gE=E'$ and $gF=F'$ if and only if $(E,F)$ and $(E',F')$ have same principal angles. Moreover, the metric $d(E,F)=\sqrt{\sum\theta_i^2}$ is the metric (up to a scalar constant) of the Riemannian symmetric space of compact type $\O(n)/\O(p)\times\O(n-p)$.\\

We introduce a similar notion of \emph{hyperbolic principal angles} between two elements of $X_p(\K)$. This allow to recover the metric since if $(\alpha_1,\dots,\alpha_p)$ are the hyperbolic principal angles between $E,F\in X_p(\K)$ then the distance between them is $\sqrt{\sum\alpha_i^2}$ (up to a scalar factor) and the family of hyperbolic principal angles is a complete invariant of pairs in $X_p(\K)$ under the action of Isom$(X_p(\K))$. Moreover, we show in the real case the following characterization of isometries of $X_p(\R)$.
\begin{theo}\label{iso}Let $g$ be a map from $X_p(\R)$ to itself. The following are equivalent :
\begin{enumerate}[(i)]
\item $g$ is an isometry.
\item $g$ preserves hyperbolic principal angles. 
\item There exists $h\in\O(p,\infty)$ such that $g=\pi(h)$.
\end{enumerate}
\end{theo}

\begin{cor}\label{ciso}The isometry group of $X_p(\R)$ is $\mathbf{PO}(p,\infty)=\O(p,\infty)/\{\pm I\}$.
\end{cor}

We note that if $p=1$ then one recover the infinite dimensional (real) hyperbolic space, which appears already in \cite{MR0000172}, for example,  and more recently in \cite{MR2152540}. This space has also some links with Kähler groups and the Cremona group (see \cite{cantat} and \cite{Delzant:2010fk} for example).\\

Since parabolic groups of finite-dimensional Riemannian symmetric spaces of non-compact type play a key role in the theory, we study parabolic groups of $X_p(\R)$ and show they are in correspondence with isotropic flags (see Proposition \ref{parabolic}).
\subsection{Furstenberg Maps}
An important step in Margulis' superrigidity is to construct \emph{Furstenberg maps}. The analogue in our infinite-dimensional setting is the main result of this paper :
\begin{theo}\label{fm}Let $G$ be locally compact second countable group and $B$ a $G$-boundary. For any continuous and non-elementary action of $G$ on $X_p(\K)$ there exists a measurable $G$-map $\varphi\colon B\to \bo X_p(\K)$.
\end{theo}
The notion of $G$-boundary is defined in section \ref{G-boundary}. Such maps $\varphi$ are called \emph{Furstenberg maps} and we hope they will allow us to obtain rigidity statements for a large class of actions of locally compact groups on $X_p(\K)$.\\

In order to obtain the previous theorem we use the notion of \emph{measurable field of CAT(0) spaces} and in particular we show a similar statement to the principal theorem in  \cite{MR1645958}. 
\begin{theo}\label{AB}Let $G$ be a locally compact second countable group and $\Omega$ an ergodic $G$-space such that $G\action \Omega$ is amenable. Let $\X$ be a measurable field of Hadamard spaces of finite telescopic dimension.\\ If $G$ acts on $\X$ then there is an invariant section of the boundary field $\bo\X$ or there exists an invariant Euclidean subfield of $\X$.
\end{theo}
For a precise meaning of terms used in this theorem, we refer to section \ref{sfm}. A result close to Theorem \ref{AB} was obtained by M. Anderegg and P. Henry (see \cite[Theorem 1.1]{Anderegg:2011fk}). M. Anderegg also used measurable fields of CAT(0) spaces to show existence of Furstenberg maps in the case of spaces of rank less than 3 (see  \cite[Theorem 5.2.1]{Anderegg}).

\begin{ack}We thank Pierre de la Harpe and Nicolas Monod for useful conversations and comments about this work.
\end{ack}
\part{Geometry}

\section{Symmetric spaces of infinite dimension}
\subsection{\texorpdfstring{$\CAT(0)$ spaces}{CAT(0) spaces}}
We recall know facts about $\CAT(0)$ spaces and introduce notations. We refer to  \cite{MR1744486} for the general theory. A metric space $(X,d)$ is a $\CAT(0)$ space if it is a geodesic space and if for every $x,y,z\in X$ and a midpoint $m$ between $y$ and $z$, the Bruhat-Tits inequality holds :
\begin{equation}\label{BT}d(x,m)^2\leq1/2\left(d(x,y)^2+d(x,z)^2\right)-1/4\ d(y,z)^2.
\end{equation}
A subspace $Y$ of a $\CAT(0)$ space $X$ is \emph{convex} if for every $x,y\in Y$ the unique geodesic segment $[x,y]$ between $x$ and $y$ is included in $Y$ and $Y$ is said to be \emph{Euclidean} if $Y$ is isometric to some $\R^n$.  If $\Delta$ is a geodesic triangle in $X$ with vertices $x,y,z$, a \emph{comparison triangle} is a Euclidean triangle $\overline{\Delta}$ of vertices $\overline{x},\overline{y},\overline{z}$ with same length of sides as $\Delta$. For three points $x,y,z$ with $x\neq y$ and $x\neq z$, the \emph{comparison angle} $\cangle_x(y,z)$ is the Euclidean angle in any comparison triangle between $\overline{y}$ and $\overline{z}$ at $\overline{x}$. If  $y'$ and $z'$ are points of $X$ respectively on the geodesic segments $[x,y]$ and $[x,z]$, the angle $\cangle_x(y',z')$ is non-increasing when $y'\to x$ and $z'\to x$. The limit $\angle_x(y,z)$ is called the \emph{Alexandrov angle} between $y$ and $z$ at $x$.\\

If $Y$ is a subset of metric space $X$, its \emph{diameter} is $\sup_{x,y\in Y}d(x,y)$. The subset $Y$ is said to be \emph{bounded} if it has finite diameter. In this case its \emph{circumradius}, rad$(Y)$ is $\inf\{r\geq0|\ \exists x\in X,\ Y\subset B(x,r)\}$ and a point $x\in X$ such that the closed ball $\overline{B}(x,$rad$(Y))$ contains $Y$ is called a \emph{circumcenter} of $Y$.  In a complete CAT(0) space, every bounded subspace has a unique circumcenter and if moreover $Y$ is closed and convex then its circumcenter belongs to it.\\

If $X$ is $\CAT(0)$ space, $Y$ a complete convex subspace of $X$ and $x$ a point of $X$ then there exists a unique point $y\in Y$ such that $d(x,y)=\inf\{ d(x,z)|\ z\in Y\}$. This point, denoted by $\pi_Y(x)$, is called the \emph{projection} of $x$ onto $Y$ and the following angle property holds
\begin{equation}\forall z\in Y,\ \angle_y(x,z)\geq \pi/2.\end{equation}

Let $Y,Y'$ be two closed convex of a complete CAT(0) space of circumradii $r,r'$ and circumcenters $c,c'$. If $Y\subset Y'$ then \cite[lemma 11] {MR2219304}
\begin{equation}\label{cc} d(c,c')^2\leq 2(r'^2-r^2).\end{equation}
\begin{prop}[Theorem 14 in \cite{MR2219304}]\label{monod14} Any filtering (for the reverse order associated with inclusion) family of closed convex bounded subspaces of a complete CAT(0) space has a non-empty intersection.
\end{prop}
\begin{proof} We recall the proof in the case of a sequence. Let $(X_n)$ be such a sequence and let $r_n$ and $c_n$ be the circumradius and circumcenter of $X_n$.  The sequence $(r_n)$ is non-increasing and non-negative. Thus this is a convergent sequence and inequality (\ref{cc}) shows $(c_n)$ is a Cauchy sequence. The limit is a point of $\cap X_n$.
\end{proof}
A usefull geometric object associated with a $\CAT(0)$ space $X$ is its \emph{boundary at infinity}. Two geodesic rays $\rho,\rho'\colon\R^+\to X$ are \emph{asymptotic} if their images are at bounded Hausdorff distance. The boundary at infinity $\bo X$ of $X$ is the set of classes of asymptotic rays. If $X$ is a complete CAT(0) space, $x\in X$ and $\xi,\xi'\in \bo X$ then there exist unique geodesic rays $\rho,\rho'$ such that $\rho(0)=\rho'(0)=x$ and $\rho,\rho'$ are respectively in class $\xi,\xi'$. The angle $\angle_x(\xi,\xi')$ between $\xi,\xi'$ at $x$ is $\angle_x(\rho(t),\rho(t'))$ for any $t,t'>0$ and $\angle(\xi,\xi')$ is $\sup_{x\in X}\angle(\xi,\xi')$. The map $(\xi,\xi')\mapsto \angle(\xi,\xi')$ is a metric on $\bo X$ called the \emph{angular metric} and the length metric (see \cite[Definition I.3.3]{MR1744486}) associated with $\angle$ is called the \emph{Tits metric} on $\bo X$. If $(x_n)$ is a sequence of points, one says that $x_n$ converges to $\xi\in \bo X$ if for any $x\in X$ and any $r>0$ the intersection $[x,x_n]\cap\overline{B(x,r)}$ converges to $\rho([0,r])$ for the Hausdorff distance where $\rho$ is the geodesic from $x$ in the class $\xi$. If $\xi,\eta\in\bo X$, $x\in X$, $x_n\to\xi$ and $y_n\to\eta$   then \cite[Lemma II.9.16]{MR1744486}
\begin{equation}\label{sc}
\liminf_{n\to\infty}\cangle_x(x_n,y_n)\geq\angle(\xi,\eta).
\end{equation}

If $\xi$ is a point at infinity of a complete CAT(0) space $X$, the \emph{Busemann function} $(x,y)\mapsto\beta_\xi(x,y)$ is defined by $\beta_\xi(x,y)=\lim_{t\to\infty}d(x,\rho(t))-t$ where $\rho$ is the geodesic ray from $y$ is the class $\xi$. Busemann functions verify the cocycle relation $\beta_\xi(x,z)=\beta_\xi(x,y)+\beta_\xi(y,z)$ for all $x,y,z\in X$. If $y$ is fixed, we call also Busemann function the function $x\mapsto \beta_\xi(x)=\beta_\xi(x,y)$. For two different base points, the cocycle relation shows the associated Busemann functions differ by a constant. The following relation for $\xi\in \bo X$, $x,y\in X$ and $\rho$ geodesic from $x$ to $\xi$ is known as the  ``asymptotic angle formula'' \cite[Section 2]{MR2574740}.
\begin{equation}\label{aaf}
\lim_{t\to\infty}\cos\left(\cangle(\rho(t),y)\right)=\frac{\beta_\xi(x,y)}{d(x,y)}
\end{equation}

\subsection{Riemannian Geometry} For  a general treatment of Riemannian geometry (in finite or infinite dimension) we refer to \cite{MR1666820} or \cite{MR1330918}. The image to have in mind is that Riemannian manifolds of infinite dimension are constructed the same  way as finite dimensional ones except that tangent spaces are Hilbert spaces instead of Euclidean ones. Let $(M,g)$ be a Riemannian manifold. As in finite dimension, one can define the Riemann tensor, the sectional curvature and the exponential map. Moreover, a complete simply connected Riemannian manifold with non-positive sectional is a complete CAT(0) space and the exponential map at any point is a diffeomorphism (see \cite[chapter XII]{MR1666820}).\\

Let $\mathcal{H}_\R$ be a real Hilbert and $\mathbf{O}(\mathcal{H}_\R)$ its orthogonal group. We define $\mathbf{L}^2(\mathcal{H}_\R)$ to be the space of all Hilbert-Schmidt operators of $\mathcal{H}_\R$. We set $\mathbf{GL}^2(\mathcal{H}_\R)$ to be the group of invertible operators that can be written $I+M$ where $I$ is the identity and $M\in \mathbf{L}^2(\mathcal{H}_\R)$. We also set $\mathbf{O}^2(\mathcal{H}_\R)=\mathbf{O}(\mathcal{H}_\R)\cap\mathbf{GL}^2(\mathcal{H}_\R)$, $\mathbf{S}^2(\mathcal{H}_\R)$ the closed subspace of symmetric operators in  $\mathbf{L}^2(\mathcal{H}_\R)$ and $\mathbf{P}^2(\mathcal{H}_\R)$ the cone of symmetric positive definite operators in $\mathbf{GL}^2(\mathcal{H}_\R)$.\\

Then $\mathbf{P}^2(\mathcal{H}_\R)$ identifies with $\mathbf{GL}^2(\mathcal{H}_\R)/\mathbf{O}^2(\mathcal{H}_\R)$. The exponential map $\exp\colon\mathbf{S}^2(\mathcal{H}_\R)\to\mathbf{P}^2(\mathcal{H}_\R)$ is a diffeomorphism. The space $\mathbf{P}^2(\mathcal{H}_\R)$ is actually a Riemannian manifold. The metric at $I$ is given by $<X,Y>=$Trace$(^tXY)$ and it has non-positive sectional curvature. Then it is a complete Cartan-Hadamard manifold. This is a Riemannian symmetric space and the symmetry at $I$ is given by $G\mapsto G^{-1}$.  Actually, this is the most natural generalization of the finite dimensional Riemannian symmetric space SL$_n(\R)/$SO$_n(\R)$ and already appeared in \cite{MR0476820} and \cite{MR2373944}.\\

 As observed in \cite[Theorem F]{MR2373944} the Riemannian symmetric space GL$_n(\R)/$O$_n(\R)$ embeds isometrically in  $\mathbf{P}^2(\mathcal{H}_\R)$. Fix a Hilbert base of $\mathcal{H}_\R$ and identify $\R^n$ with the subspace spanned be the $n$ first vectors of the Hilbert base. This gives an embedding GL$_n(\R)\hookrightarrow\mathbf{GL}^2(\mathcal{H}_\R)$. An operator $G\in $GL$_n(\R)$ is extended by the identity on the orthogonal of $\R^n\subset\mathcal{H}_\R$. This induces an embedding GL$_n(\R)/$O$_n(\R)\hookrightarrow\mathbf{P}^2(\mathcal{H}_\R)$ and (up to a scalar factor) this is an isometric embedding. With the previous identification, $\mathbf{P}^2(\mathcal{H}_\R)$ is the closure of the union $\bigcup_n $GL$_n(\R)/$O$_n(\R)$.  \\
 
 As observed by P. de la Harpe, the characterization \cite[Corollary of theorem I]{MR0069829} of totally geodesic subspaces of SL$_n(\R)/$SO$_n(\R)$ obtained by Mostow is also true in the infinite dimensional case. We recall that $\mathbf{L}^2(\mathcal{H}_\R)$ is  a Lie algebra with Lie bracket $[X,Y]=XY-YX$ and also a Hilbert space with scalar product $<X,Y>=$Trace$(^tXY)$. Then a \emph{Lie triple system} of $\mathbf{L}^2(\mathcal{H}_\R)$ is a closed linear subspace $\mathfrak{p}$ such that for all $X,Y,Z\in\mathfrak{p}$, $[X,[Y,Z]]\in\mathfrak{p}$. A totally geodesic subspace of a geodesically complete Riemannian manifold $X$ is a closed submanifold $Y$ that such for any point $y\in Y$ and any vector $v\in T_yY$ the whole geodesic with initial vector $v$, is included in $Y$.
 
 \begin{lem}[Proposition III.4 in \cite{MR0476820}] Let $\mathfrak{p}$ be a Lie triple system of $\mathbf{S}^2(\mathcal{H}_\R)$ then $\exp(\mathfrak{p})$ is a totally geodesic subspace of $\mathbf{P}^2(\mathcal{H}_\R)$. Moreover, all totally geodesic of $\mathbf{P}^2(\mathcal{H}_\R)$ which contains $I$ are obtained this way.
 \end{lem}

Let $\mathcal{H}$ be a $\mathbb{K}$-Hilbert space as in the introduction and $\mathcal{H}_\R$ the underlying real Hilbert space associated with $\mathcal{H}$. We set $\mathbf{L}(\mathcal{H})$ (respectively $\mathbf{GL}(\mathcal{H})$) to be the subset of all (respectively invertible) bounded $\K$-linear operators on $\mathcal{H}$.  This is $\mathbf{GL}(\mathcal{H}_\R)$ if $\K=\R$, the subspace of $\mathbf{GL}(\mathcal{H}_\R)$ of operators that commute with multiplication by $i$ if $\K=\C$ and the subspace of $\mathbf{GL}(\mathcal{H}_\R)$  of operators that commute with multiplications by $i,j,k$ if $\K=\H$.\\

For $A\in\mathbf{L}(\mathcal{H})$ there is a unique operator $A^*\in\mathbf{L}(\mathcal{H})$ called the \emph{adjoint operator} of $A$ such that for all $x,y\in\mathcal{H}$, $<Ax,y>=<x,A^*y>$.
\subsection{Riemannian structure on $X_p(\K)$}
We set $\O_{p,\infty}(\K)$ to be the subgroup of $\mathbf{GL}(\mathcal{H})$ of elements $G$ such that $G^*\Phi G=\Phi$. A more natural  (but less uniform) way to denote these groups could be respectively $\O(p,\infty)$, $\mathbf{U}(p,\infty)$ and $\mathbf{Sp}(p,\infty)$. We also set $\O^2_{p,\infty}(\K)$ to be the closed subgroup $\O_{p,\infty}(\K)\cap\mathbf{GL}^2(\mathcal{H}_\R)$ of $\mathbf{GL}^2(\mathcal{H}_\R)$.\\

The group $\O_{p,\infty}(\K)$ acts naturally on $X_p(\K)$ and the stabilizer of $E_0$  is $\O_{p,\infty}(\K)\cap\O_\infty(\K)$ (where $\O_\infty(\K)$ is the orthogonal group of $\mathcal{H}$). Indeed an element that fixes $E_0$ induces an orthogonal operator of it and another orthogonal transformation of its orthogonal  for $B_p$, which it is also its orthogonal for the scalar product. Thus the stabilizer of $E_0$ is exactly $\O_p(\K)\times\O_\infty(\K)$. Once again it could be more natural to write this group $\O(p)\times\O(\infty)$, $\mathbf{U}(p)\times\mathbf{U}(\infty)$ of $\mathbf{Sp}(p)\times\mathbf{Sp}(\infty)$ depending if $\K=\R,\C$ or $\mathbb{H}$.\\

Let $E\in X_p(\K)$. Witt's theorem \cite[Theorem 1.2.1]{MR0150210} implies that there exists an element $g$ of $\O_{p,\infty}(\K)$ such that $gE_0=E$. This shows  $\O_{p,\infty}(\K)$ acts transitively on $X_p(\K)$. It shows a little bit more. Let $\O^F_{p,\infty}(\K)$ (respectively $\O^F_{\infty}(\K)$) be the subgroup of $\O_{p,\infty}(\K)$ (respectively of $\O_{\infty}(\K)$) of operators that can be written $I+M$ where $M$ is a finite rank operator. Thus $\O^F_{p,\infty}(\K)$ acts transitively on $X_p(\K)$.\\

Actually, $X_p(\K)$ can be identified with different quotient spaces
\begin{align*}
X_p(\K)&\simeq \O_{p,\infty}(\K)/\O_p(\K)\times\O_\infty(\K),\\
&\simeq \O^2_{p,\infty}(\K)/\O_p(\K)\times\O^2_\infty(\K),\\
&\simeq \O^F_{p,\infty}(\K)/\O_p(\K)\times\O^F_\infty(\K).
\end{align*}\\

Let $\mathfrak{o}_{p,\infty}(\K)$ be the subspace of $\mathbf{L}(\mathcal{H})$ of operators $A$ such that $A^*\Phi+\Phi A=0$. We also denote by $\mathfrak{o}^2_{p,\infty}(\K)$ the intersection $\mathfrak{o}_{p,\infty}(\K)\cap\mathbf{L}^2(\mathcal{H}_\R)$.

\begin{prop}\label{tgs} The space $X_p(\K)$ embeds  as a totally geodesic subspace of $\mathbf{P}^2(\mathcal{H}_\R)$.
\end{prop}
Before proceeding with the proof of this proposition we recall what happens in finite dimension. We use the notations of section \ref{iGeo}. We suppose that $\mathcal{H}$ has dimension $n=p+q$ where $1\leq p<q$ and then we define $X_{p,q}(\K)=\left\{E\in\mathcal{G}_p\left|\ B_p|_{E\times E}\ \mathrm{is\ positive\ definite}\right.\right\}$. If $\O_{p,q}(\K)$ is the orthogonal group of $Q_p$ then $X_{p,q}(\K)\simeq\O_{p,q}(\K)/\O_{p}(\K)\times\O_{q}(\K)$ and it is an irreducible Riemannian symmetric space of non-compact type. Moreover $\O_{p,q}(\K)$ is a reductive subgroup of $\mathbf{GL}(\mathcal{H}_\R)\simeq $GL$_m(\R)$ where $m=\dim_\R(\K)$. We recall that a reductive subgroup of $GL_m(\R)$ is a closed Lie subgroup stable under transposition. See \cite[II.10.57]{MR1744486} for example.\\

The Lie algebra of $\O_{p,q}(\K)$ is $\mathfrak{o}_{p,q}(\K)=\{A\in\mathbf{L}(\mathcal{H})|\ A^*\Phi+\Phi A=0\}$ and $X_{p,q}(\K)$ identifies with a totally geodesic subspace of GL$_m(\R)/$O$_m(\R)$, which is the image of symmetric elements in $\mathfrak{o}_{p,q}(\K)$ by the exponential map. See \cite[Theorem II.10.57]{MR1744486} for more details. \\

The tangent space at $I$ of GL$_m(\R)/$O$_m(\R)$ is the space of symmetric operators and we endow it with the scalar product $<X,Y>=$Trace$(^tXY)$. Up to a scalar factor it coincides with the scalar product coming from the Killing form. The advantage of this scalar product is that GL$_m(\R)/$O$_m(\R)$ embeds isometrically and totally geodesically in GL$_{m'}(\R)/$O$_{m'}(\R)$ for $m\leq m'$. \\

Choose a base $(e_i)$ of $\mathcal{H}$ such that $E_0$ is spanned by the $p$ first vectors.  We naturally identify $\O_{p,q}(\K)$  with the subgroup of $\O_p(\K)$ that acts on the span of $e_1,\dots, e_{p+q}$ and is the identity on the orthogonal of this subspace.  We do obviously the same for $\mathfrak{o}_{p,q}(\K)$.
\begin{proof}[Proof of Proposition \ref{tgs}] The space $\mathfrak{p}=\mathfrak{o}_{p,\infty}(\K)\cap\mathbf{S}^2(\mathcal{H}_\R)$ is a Lie triple system because it is the closure of $\cup_{q>p}(\mathfrak{o}_{p,q}(\K)\cap\mathbf{S}^2(\mathcal{H}_\R))$. Let $X=\exp(\mathfrak{p})$, $G$ be the subgroup of $\mathbf{GL}^2(\mathcal{H}_\R)$ generated by $\exp(\mathfrak{o}^2_{p,\infty}(\K))$ and $K=G\cap\mathbf{O}(\mathcal{H}_\R)$. Then $X$ is totally geodesic subspace of $\mathbf{P}^2(\mathcal{H}_\R)$, $G$ acts transitively on $X$ and $X\simeq G/K$.\\

Indeed $G$ is a subgroup of $\O^2_{p,\infty}(\K)$ which contains  $\O^F_{p,\infty}(\K)$. Thus $G$ acts transitively on $X_p(\K)$ and the stabilizer of $E_0$ is exactly $K$ then $X_p(\K)\simeq G/K\simeq X$. 
\end{proof}
\begin{rem}\label{plongement}Let $d=\dim_\R(\K)$. If one considers $\mathcal{H}_\R$, the symmetric bilinear form Re$(B_p)$ is actually $B_{dp}$ and elements of $X_p(\K)$ are also elements of $X_{dp}(\R)$ considered as real vector subspaces. Thus, $X_p(\K)$ can be identified with a subset of $X_{dp}(\R)$. Moreover $\mathfrak{o}_{p,\infty}(\K)$ can be identified with a Lie subalgebra of $\mathfrak{o}_{dp,\infty}(\R)$ and thus $X_p(\K)$ can be identified with a totally geodesic subspace of $X_dp(\R)$.
\end{rem}

The embedding provided in Proposition \ref{tgs} allows us to endow $X_p(\K)$ with the pull back of the metric on $\mathbf{P}^2(\mathcal{H}_\R)$. Now, $X_p(\K)$ will always be endowed with this metric. Actually, the previous embedding shows that $X_p(\K)$ is a Riemannian symmetric space of non-positive sectional curvature but we will retain less information.
\begin{cor}The space $X_p(\K)$ is a separable complete CAT(0) space.
\end{cor}
\begin{prop}\label{fc}For all finite configuration of points, geodesics, points at infinity and Euclidean subspaces of finite dimension, there is a closed totally geodesic space $Y$ of $X_p(\K)$ that contains the elements of the configuration and that is isometric to some $X_{p,q}(\K)$ with $q\geq p$.\\
Moreover, every isometry of $Y$ coming from $\O_{p,q}(\K)$ is the restriction of an isometry of $X_p(\K)$.
\end{prop}
\begin{proof}Indeed, it suffices to show the result for a finite number of points in $X_p(\K)$ because a Euclidean subspace of finite dimension is completely determined by a finite number of geodesic lines, a geodesic line is completely determined by two different points on it and a point at infinity is determined by a geodesic ray pointing toward it.\\

Let $E_1,\dots, E_n\in X_p(\K)$. There exists a finite dimensional subspace $\mathcal{H}_0$ of $\mathcal{H}$ that contains $E_0,E_1,\dots,E_n$. Let $p+q$ be the dimension of $\mathcal{H}_0$.  Then $E_0,E_1,\dots,E_n$ lies in some $X_{p,q}(\K)$ isometrically and totally geodesically embedded in $X_p(\K)$. Moreover the only geodesic line through $E_i$ and $E_j$ is contained in this $X_{p,q}(\K)$.
\end{proof}
\section{Metric approach}
\subsection{Hyperbolic principal angles}
Let $E,F$ be two elements of $X_p(\K)$. We will define successively the family of their hyperbolic principal angles.\\

Let $c_1=\sup\{B(x,y)|\ x\in E,\ y\in F,\ Q(x)=Q(y)=1\}$. Since $Q|_E$ and $Q|_F$ are positive definite, there exist $x_1\in E$ and $y_1\in F$ such that $Q(x_1)=Q(y_1)=1$ and $B(x_1,y_1)=c_1$. Suppose $c_i,x_i,y_i$ are defined for $i=1,\dots, l<p$, we define $E_l=\{x_1,\dots,x_l\}^{\bot_Q}\cap E$, $F_l=\{y_1,\dots,y_l\}^{\bot_Q}\cap F$ and $c_{l+1}=\sup\{B(x,y)|\ x\in E,\ y\in F,\ Q(x)=Q(y)=1\}$. We choose once again $x_{l+1}\in E_l$ and $y_{l+1}\in F_l$ such that $Q(x_{l+1})=Q(y_{l+1})=1$ and $B(x_{l+1},y_{l+1})=c_{l+1}$.
\begin{rem}In case $p=1$, the reverse Schwartz inequality (see \cite[II.10.3]{MR1744486}) shows that $|B(x_1,y_1)|\geq 1$ and it is possible to define the hyperbolic angle between $x_1$ and $y_1$ by $\alpha_1$=arccosh$(|B(x_1,y_1)|)$. However if $p\geq 2$ and  $x,y\in\mathcal{H}$ such that $Q(x)=Q(y)=1$ it is possible that $|B(x,y)|<1$ and it is impossible to define a hyperbolic angle between $x$ and $y$.
\end{rem}
If $E\in X_p(\K)$ one can define the \emph{orthogonal projector} $P_E$ on $E$ with respect to $Q$. This is the unique linear operator $P$ such that $P|_E=$Id$_E$, $P|_{E^{\bot_Q}}=0$.\\

For $x\in E$ and $y\in F$, $B(x,y)=B(P_F(x),y)=B(x,P_E(y))$. If $B$ is a bilinear symmetric form, two orthonormal bases $(x_i)$ and $(y_i)$ of subspaces $E$ and $F$ are \emph{biorthogonal} if for $i\neq j$, $B(x_i,y_j)=0$.
\begin{lem} The p-uple $(c_1,\dots,c_p)$ does not depend on choices of $x_i$ and $y_i$. Moreover $(x_i)$ and $(y_i)$ are biorthogonal bases of $E$ and $F$ with respect to $B$. 
\end{lem}
\begin{proof} We first show $(x_i)$ and $(y_i)$ are biorthogonal bases. It is clear that $(x_i)$ and $(y_i)$ are bases of respectively $E$ and $F$. Thus, it suffices to show that for $1\leq i\leq p$, $P_F(x_i)=c_iy_i$ and $P_E(y_i)=c_ix_i$. Fix $1\leq i\leq p$ and suppose this is true for $1\leq j <i$. Since $B(x_i,x_j)=0$ for all $1\leq j <i$, $B(x_i,P_E(y_j))=0$ and $B(P_E(x_i),y_j)=0$ for all $1\leq j <i$. So, $P_F(x_i)\in F_i$ and $P_F(x_i)=c_i y_i$. Symmetrically, $P_F(y_i)=c_i x_i$.\\

Now, consider the operator $P_EP_FP_E$. In an orthogonal base starting with $(x_i)$ the matrix of this map is diagonal with diagonal entries $c_1^2,\dots,c_p^2,0,0,\dots$. Thus, the non-trivial eigenvalues are exactly the $c_i^2$ and do no depend on the base $(x_i)$.
\end{proof}
We can now define hyperbolic principal angles between $E$ and $F$. For $1\leq i\leq p$, $P|_F(x_i)=c_iy_i$ so $Q(x_i)=c_i^2Q(y_i)+Q(x_i-c_iy_i)$. Since $Q(x_i)=Q(y_i)=1$ and $Q$ is negative definite on $F^{\bot_Q}$, $c^2_i=1-Q(x_i-c_iy_i)\geq 1$. We define $\alpha_i=$arccosh$(c_i)$ for $1\leq i\leq p$. The non-increasing family $(\alpha_i)$ of non-negative real numbers is called the family of hyperbolic principal angles between $E$ and $F$.
\begin{prop}Let $E,F,E',F'\in X_p(\K)$. There exists $g\in\O_{p,\infty}(\K)$ such that $g(E)=E'$ and $g(F)=F'$ if and only if $E,F$ and $E',F'$ have same families of hyperbolic principal angles.
\end{prop}
\begin{proof}From the definition of hyperbolic principal angles, it is clear that the existence of such $g$ implies $E,F$ and $E',F'$ have same families of hyperbolic principal angles.\\

Suppose $E,F$ and $E',F'$ have same families of hyperbolic principal angles. We choose $(x_i)_i,(y_i)_i$ biorthogonal bases for  $E,F$ and  $(x'_i)_i,(y'_i)_i$ biorthogonal base for $E',F'$. We set $u_i=y_i-P_E(y_i)$ and $u'_i=y'_i-P_{E'}(y'_i)$. Let $G,G'$ the span of, respectively, $E\cup F$ and $E'\cup F'$. Then $\{x_i\}\cup\{u_i,\ u_i\neq0\}$ and $\{x'_i\}\cup\{u'_i,\ u'_i\neq0\}$ are orthogonal bases of $G$ and $G'$. Since the restriction of $Q_p$ to $G$ and $G'$ are equivalent, Witt's theorem yields $g\in\O_{p,\infty}(\K)$ such that $gx_i=x'_i$ and $gu_i=u'_i$ for $1\leq i\leq p$.
\end{proof}
\begin{rem}\label{aphrcq} We have seen in remark \ref{plongement} $X_p(\K)$ embeds as a totally geodesic subspace of $X_{dp}(\R)$ where $d=\dim_\R(\K)$. So it is a natural to try to understand what is the link between hyperbolic principal angles between $E$ and $F$ in $X_p(\K)$ and the hyperbolic angles between $E$ and $F$  considered as elements of $X_p(\R)$. 
Actually if $(x_1,\dots,x_p)$ and $(y_1,\dots,y_p)$ are biorthogonal bases associated with $E$ and $F$ (as $\K$-vector spaces) then $(x_1,ix_1,\dots,x_p,ix_p)$ and $(y_1,iy_1,\dots,y_p,iy_p)$ are biorthogonal bases associated with $E$ and $F$ as $\R$-vector spaces if $\K=\C$ and $(x_1,ix_1,jx_1,kx_1\dots,x_p,ix_p,jx_p,kx_p)$ and $(y_1,iy_1,jy_1,ky_1\dots,y_p,iy_p,jy_p,ky_p)$ are biorthogonal bases associated with $E$ and $F$ as $\R$-vector spaces if $\K=\H$. If $(\alpha_1,\dots,\alpha_p)$ is the family of hyperbolic  principal angles between $E$ and $F$ as elements of $X_p(\K)$ then $(\alpha_1,\alpha_1,\dots,\alpha_p,\alpha_p)$ is the family of principal hyperbolic angles between $E$ and $F$ as elements of $X_p(\R)$ if $\K=\C$ or  $(\alpha_1,\alpha_1,\alpha_1,\alpha_1,\dots,\alpha_p,\alpha_p,\alpha_p,\alpha_p)$ is the family of hyperbolic principal angles between $E$ and $F$ as elements of $X_p(\R)$ if $\K=\H$.
\end{rem}

\begin{prop}Let $E,F\in X_p(\K)$ and $(\alpha_i)$ their family of hyperbolic principal angles. Then \[d(E,F)^2=2\,\mathrm{dim}_\R(\K)\sum_{i=1}^p\alpha_i^2.\]

\end{prop}
\begin{proof}We begin by the case $\K=\R$. Thanks to Proposition \ref{fc}, we can consider that $E,F$ are points in some $X_{p,q}(\R)$.\\

We recall some facts about $X_{p,q}(\R)$ and $\O(p,q)$. Let $n=p+q$ and $\Phi$ be the diagonal matrix of GL$_n(\R)$ with first $p$ occurrences of 1 and $q$ occurrences of $-1$. The group $\O(p,q)$ is then the subgroup
\[\left\{M\in\mathrm{GL}_n(\R);\ ~^\mathrm{t}M\Phi M=\Phi\right\}\]
and the Lie algebra of $\O(p,q)$ is
\[\mathfrak{o}(p,q)=\left\{H\in\mathrm{M}_n(\R);\ ~^\mathrm{t}H\Phi+\Phi H=0\right\}.\]
Any element $H\in\mathrm{M}_n(\R)$ is in $\mathfrak{o}(p,q)$ if and only if 
\[H=\left[
\begin{array}{cc}
A&B\\
~^\mathrm{t}B&C
\end{array}\right]\]
where $A\in\mathrm{M}_p(\R)$ and $C\in\mathrm{M}_q(\R)$ are skew-symmetric and $B$ is any matrix in $\mathrm{M}_{p,q}(\R)$. The space $\mathrm{M}_{p,q}(\R)$ can be identified with the tangent space of $X_{p,q}(\R)$ at $E_0$ and in this case the exponential map is
\[B\mapsto\exp\left(\left[
\begin{array}{cc}
0&B\\
~^\mathrm{t}B&0
\end{array}\right]\right)\cdot E_0.\]
Let $\mathfrak{p}$ be the subspace of symmetric elements in $\mathfrak{o}_{p,q}(\R)$ (which we identify with $\mathrm{M}_{p,q}(\R)$). A maximal abelian subspace of $\mathfrak{p}$ is given, for example, by the set of matrices $H_\lambda$ where
\begin{equation}\label{eqhl}
H_\lambda=\left[\begin{array}{ccc}
0&D_\lambda&0\\
D_\lambda&0&0\\
0&0&0
\end{array}\right]
\end{equation}
and $D_\lambda$ is the diagonal matrix with diagonal $\lambda=(\lambda_1,\dots,\lambda_p)\in\R^p$. The Killing form on $\mathfrak{o}(p,q)$ is $K(H,H')=(p+q-2)$Trace$(^tHH')$ but the (natural) choice of scalar product on $\mathbf{S}_\infty^2(\R)$ correspond to $K'=\frac{1}{p+q-2}K$. With this scalar product 
\[K'(H_\lambda,H_\lambda)=2\sum_{i=1}^p\lambda_i^2.\]
Moreover a computation shows that 
\[\exp(H_\lambda)=\left[
\begin{array}{ccc}
\cosh(\lambda)&\sinh(\lambda)&0\\
\sinh(\lambda)&\cosh(\lambda)&0\\
0&0&\mathrm{I}_{q-p}
\end{array}\right]\]
where $\cosh(\lambda)$ (respectively $\sinh(\lambda)$) is the diagonal matrix with diagonal entries $\cosh(\lambda_1),\dots,\cosh(\lambda_p)$ (respectively $\sinh(\lambda_1),\dots,\sinh(\lambda_p))$ and thus 
\begin{equation}\label{sum}
d(\exp(H_\lambda)x_0,x_0)^2=K'(H_\lambda,H_\lambda)=2\sum_{i=1}^p\lambda_i^2.
\end{equation}
Since $\O(p,q)$ acts transitively on $X_{p,q}$ preserving both distance and hyperbolic principal angles, we can also suppose that $E=E_0$ and $F=\exp(H_\lambda)E_0$ for some $\lambda\in\R^p$. Then, it suffices to remark that the set of hyperbolic principal angles are $\{|\lambda_i|\}$. Equation (\ref{sum}) concludes the real case.\\

Now, if $\K$ is $\C$ or $\H$, thanks to remark \ref{plongement}, it suffices to understand the distance between $E$ and $F$ as elements of $X_{dp}(\R)$ where $d=\dim_\R(\K)$. Then the above treatment of the real case and remark  \ref{aphrcq} show that the desired formula for the distance between $E$ and $F$ holds.
\end{proof}

\subsection{Isometries of $X_p(\R)$}
\begin{proof}[Proof of Theorem \ref{iso}] Implications (iii)$\Rightarrow$(ii)$\Rightarrow$(i) are easy.  We show (i)$\Rightarrow$(iii). \\

Let $g\in\Isom(X_p)$. Since $\O(p,\infty)$ acts transitively on $X_p$. We can assume $gE_0=E_0$. We define the differential of $g$ at $E_0$ by $T_{E_0}g(H)=\lim_{t\to0}g\exp(tH)$ for $H\in T_{E_0}X_p\simeq L(E_0,E_0^\bot)$. This map is homogeneous and preserves the norm of $T_{E_0}X_p$. It is a classical result that $T_{E_0}g$ is a linear isometry of $T_{E_0}X_p$. We will understand $T_{E_0}g$ on the orthonormal base $(\delta_{ij})$ of $L(E_0,E^\bot_0)$ where $\delta_{ij}$ is the map $x\mapsto<x,e_i>\varepsilon_j$ for $(e_i)_{1\leq i\leq p}$ a orthonormal base of $E_0$ and $(\varepsilon_j)_{j\in\N}$ a orthonormal base of $\mathcal{H}_0$. The point $\exp(\delta_{ij})$ of $X_p$ is Span$(e_1,\dots,e_{i-1},\cosh(1)e_i+\sinh(1)\varepsilon_j,e_{i+1},\dots,e_p)$ and lies on a singular geodesic (that is a geodesic contained in at least two maximal Euclidean subspaces) of $E_0$.\\

A maximal Euclidean subspace containing $E_0$ can be written
\[P_{u,v}=\left\{\mathrm{Span}\left(\{\cosh(\lambda_i)u_i+\sinh(\lambda_i)v_i\}_{i=1..p}\right)\mid\ (\lambda_1,\dots,\lambda_p)\in\R^p\right\}\]
for $u=(u_1,\dots,u_p)$  orthonormal base of $E_0$ and $v=(v_1,\dots,v_p)$ orthonormal family of  $\mathcal{H}_0$.\\

The action of $\O(p)\times\O(\infty)$ on $L(E_0,\mathcal{H}_0)$ is given by $H\mapsto BH~^tA$ for $A\in\O(p)$ and $B\in\O(\infty)$. Thus, $\O(p)\times\O(\infty)$ acts transitively on the maximal Euclidean subspaces of $X_p$ containing $E_0$. We can suppose $g$ fixes pointwise the maximal Euclidean subspace $P_{e,\mathbf{\varepsilon}}$ where $e=(e_1,\dots,e_p)$ and $\mathbf{\varepsilon}=(\varepsilon_1,\dots,\varepsilon_p)$. If $\varepsilon'$ is obtained from $\varepsilon$ by replacing $\varepsilon_i$ by some $\varepsilon_k$ for $k>p$ then the intersection of $P_{e,\mathbf{\varepsilon}}$ and $P_{e,\mathbf{\varepsilon'}}$ is a Euclidean subspace of dimension $p-1$ that is obtained by setting $\lambda_i$ to $0$. The same property holds for $gP_{e,\mathbf{\varepsilon'}}$ and $P_{e,\mathbf{\varepsilon}}$. Thus $gP_{e,\mathbf{\varepsilon'}}$ can be written $P_{e,\mathbf{\varepsilon''}}$ where $\varepsilon''$ is obtained from $\varepsilon$ by replacing the $i$-th coordinate by a unitary vector $\varepsilon''_k$ orthogonal to $\varepsilon_1,\dots,\varepsilon_p$. \\

Since $\varepsilon_k''$ is determined by the image of the singular geodesic containing 0 and\\ $\mathrm{Span}(e_1,\dots,e_{j-1},\cosh(1)e_j+\sinh(1)\varepsilon_k,e_{(j+1)},\dots,e_p)$, $\varepsilon_k''$  depends on $\varepsilon_k$ and, a prioiri, on $i$ but does not depend on $\varepsilon_i$ for $i\neq k$. Thus, for a fixed $1\leq i\leq p$, the map   $\varepsilon_k\mapsto\varepsilon_k''$ is well defined and can be extended in a linear orthogonal map $B_i\in\O(\infty)$. So, $T_{E_0}g$ can be written $H=[h_1,\dots,h_p]\mapsto[B_1h_1,\dots,B_ph_p]$ where $h_i$ is $He_i$. It remains to show are all $B_i$ are the same map.\\

Up to  (post-) compose $g$ by Id$\times B_1^{-1}$ we can suppose that $B_1=$Id. Let $u_1,u_2$ be two orthogonal unitary vectors of $E^\bot_0$. The image  of $\mathrm{Span}(\cosh(1)e_1+\sinh(1)u_1,\dots,e_{i-1},\cosh(1)e_i+\sinh(1)u_2,e_{(i+1)},\dots,e_p)\in X_p$ is \\ $\mathrm{Span}(\cosh(1)e_1+\sinh(1)u_1,\dots,e_{i-1},\cosh(1)e_i+\sinh(1)B_iu_2,e_{(i+1)},\dots,e_p)\in X_p$. Thus, $B_iu_2$ is orthogonal to $u_1$ for all such $u_1,u_2$ and thus $B_iu_2=\pm u_2$. So, $B_i=\epsilon_i$Id with $\epsilon_i=\pm1$. Finally, if $A$ is the diagonal matrix of $\O(p)$ with diagonal entries $\epsilon_i$ then $g\cdot(A\times$Id$)=$Id$_{X_p}$.

\end{proof}
\begin{proof}[Proof of corollary \ref{ciso}] Let $\pi\colon \O(p,\infty)\to\Isom(X_p(\R))$. Thanks to Theorem \ref{iso}, $\Isom(X_p(\R))\simeq \O(p,\infty)/$ker$(\pi)$. So, it suffices to show that ker$(\pi)=\{\pm$Id$\}$. Let $g\in$ker$(\pi)$. Since $g\cdot E_0=E_0$, $g\in\O(p)\times\O(\infty)$. Let $A\in\O(p)$ and $B\in\O(\infty)$ such that $g=A\times B$. the differential of $\pi(g)$ at $E_0$ is $H\mapsto BH^t\! A$. Let $a_{ij}$ and $(b_{ij})$ the matrix coefficients of $A$ and $B$. We choose $H$ to have all matrix coefficients $0$ except the one in position $(i,j)$ which  is 1. Then  the matrix coefficient in position $(k,l)$ of $BH^t\! A$ is $a_{lj}b_{ki}$. This implies that $a_{ij}=0$ for $i\neq j$, $b_{ij}=0$ for $i\neq j$ and, for all $i\leq p$ and all $j$, $a_{ii}b_{jj}=1$. Finally, $a_{ii}=b_{jj}=a_{11}=\pm 1$ and $G=\pm$Id.
\end{proof}
\section{Telescopic dimension}
\subsection{Geometric dimension} In \cite{MR1704987}, B. Kleiner introduces a notion of dimension for  $\CAT(\kappa)$ spaces. Recall that if $X$ is a $\CAT(\kappa)$ space and $x$ a point in $X$, the space of directions $\Sigma_x X$ at $x$ is the quotient metric of the set of germs of geodesics issued from $x$, endowed with the Alexandrov angle as pseudometric. This a $\CAT(1)$ space. Let $\mathfrak{X}$ be the set of all spaces that are $\CAT(\kappa)$  for some $\kappa\in\R$. The geometric dimension is the smallest function on $\mathfrak{X}$ such that discrete spaces have dimension 0 and GeomDim($X)\geq$GeomDim$(\Sigma_xX)+1$ for all $x\in X$ and $X\in\mathfrak{X}$. 
\begin{theo}[Theorem  A in  \cite{MR1704987}]\label{car} Let $\kappa\in\R$ and $X$ $\CAT(\kappa)$ space. The following quantities are equal to the geometric dimension of $X$.\begin{itemize}
\item[$\bullet$]$\sup\{\mathrm{DimTop}(K)|\ K$ compact subset of $X\}$ where $\mathrm{DimTop}(K)$ denotes the topological dimension of $K$.
\item[$\bullet$]$\sup\{k\in\N|\ H_k(U,V)\neq0$ for some open subsets $V\subseteq U\subset X\}$. Where $H_k(U,V)$ is the $k$-th relative homology group.
\end{itemize}
\end{theo}
\subsection{Telescopic dimension}
We recall the ``probabilistic point of view" on ultrafilters. We refer to  \cite[I.5.47]{MR1744486} and references therein for more details. A  \emph{non-principal ultrafilter} $\mathcal{U}$ on the set of positive integers $\N$  is a function $\mathcal{P}(\N)\to\{0,1\}$ such that if $A,B\subseteq \N$, $A\cap B=\emptyset$ then $\mathcal{U}(A\cup B)=\mathcal{U}(A)+\mathcal{U}(B)$ and $\mathcal{U}(A)=0$ if $A$ is finite.\\

Let $\mathcal{U}$ be a non-principal ultrafilter on $\N$. A sequence $(a_n)$ of real numbers is said to converge toward $l\in\R$ with respect to $\mathcal{U}$ if for every $\epsilon>0$, $\mathcal{U}(\{n\in\N|\ |l-a_n|<\epsilon\})=1$. An important property of an ultrafilter is that every bounded sequence has a limit with respect to this ultrafilter.\\

Let $(X,d)$ be a metric space, $(\lambda_i)$ a sequence of positive numbers such that $\lambda_i\to0$, $(x_i)$ a sequence of points in $X$ and $\mathcal{U}$ a non-principal ultrafilter on $\N$. Let $X^\infty=\{(y_n)\in X^\N|\ (\lambda_id(y_i,x_i))$ is bounded$\}$. For $y=(y_i)$ and $z=(z_i)$ in $X^\infty$ we define $d^\infty(y,z)$ to be the limit with respect to $\mathcal{U}$ of the bounded sequence $(\lambda_id(y_i,z_i))$. Then $d^\infty$ is a pseudometric on $X^\infty$ and the quotient metric space associated is called the \emph{asymptotic cone} of $X$ with respect to  $\mathcal{U}$, $(\lambda_i)$ and $(x_i)$. A metric space $Y$ is called an asymptotic cone of $X$ if it is isometric to the asymptotic cone of $X$ with respect to some $\mathcal{U}$, $(\lambda_i)$ and $(x_i)$.\\

Every asymptotic cone of a $\CAT(0)$ space is a complete CAT(0) space. Following \cite{MR2558883}, a $\CAT(0)$ space $x$ has telescopic dimension less than $n\in \N$ if every asymptotic cone of $X$ has geometric dimension less than $n$ and the telescopic dimension of $X$ is the minimal $n\in\N$ such that $X$ has telescopic dimension less than $n\in \N$. The telescopic dimension can be characterized quantitatively by an asymptotic equivalent of Jung's inequality between circumradii and diameters of bounded subsets of Euclidean spaces. 
\begin{prop}[Theorem 1.3 in \cite{MR2558883}]\label{cl1.3} Let $X$ be a $\CAT(0)$ space and $n$ be a positive integer. 
The pace $X$ is of finite telescopic dimension if and only if for any $\delta>0$ there exists $D>0$ such that for any bounded subset $Y\subset X$ of diameter larger than $D$, we have 
\begin{equation}\label{Jung}\mathrm{rad}(Y)\leq \left(\delta+\sqrt{\frac{n}{2(n+1)}}\right)\mathrm{diam}(Y).\end{equation}
\end{prop}
The following theorem is the first important result about $\CAT(0)$ spaces of finite telescopic dimension.
\begin{theo}[Theorem 1.1 in \cite{MR2558883}]\label{1.1cl} Let $X$ a complete CAT(0) of finite telescopic dimension and $(X_\alpha)$ a filtering family of closed convex subspaces. If $\cap X_\alpha=\emptyset$ then $\cap\bo X_\alpha$ is not empty and has radius at most $\pi/2$ (for the angular metric). 
\end{theo}
The authors prove this theorem using gradient flows of convex functions.  We prove it without these analytic tools. Our proof use only elementary geometric facts on $\CAT(0)$ spaces and the explicit construction  of  the circumcenter at infinity will be convenient for measurability questions in section \ref{mfcs}. Similar ideas already appeared in \cite{MR1629391}.\\ 

The crucial point to show Theorem \ref{1.1cl} is to deal with nested sequences of closed convex subspaces. This is the following proposition. We show it without using gradient flow. The end of the proof of Theorem \ref{1.1cl} can be done as in \cite{MR2558883} and does not use any gradient flow.
\begin{prop}[Lemma 5.4 and 5.5 in \cite{MR2558883}]\label{cl5.4} Let $X$ a complete CAT(0) of finite telescopic dimension and $(X_i)_{i\in\N}$ a nested sequence of closed convex subspaces. If $\cap X_i=\emptyset$ then $\cap\bo X_i$ is not empty and has radius at most $\pi/2$.
\end{prop}
\begin{proof} Let $p\in X$ and $x_i$ the projection of $p$ on $X_i$. Since $\cap X_i=\emptyset$, we know that $d(p,x_i)\to\infty$ (see Proposition \ref{monod14}). We introduce the following notations : \\
\begin{itemize}
\item[$\bullet$] $N_t=\min\{i\in\N,\ d(p,x_i)\geq t\}$,
\item[$\bullet$] for $i\geq N_t$, $x_i(t)$ is the point on $[p,x_i]$ at distance $t$ from $p$,
\item[$\bullet$] $\mathcal{C}_i^t=\{x_i(t),\ i\geq N_t\}$,
\item[$\bullet$] $D_t=\diam(\mathcal{C}_i^t)$,
\item[$\bullet$] for $i\geq N_t$, $c_i^t$ is the circumcenter of $\{x_j(t),\ j\geq i\}$ and
\item[$\bullet$] $r_i^t=\mathrm{rad}\{x_j(t),\ j\geq i\}$.
\end{itemize} 

\begin{figure*}
\begin{tikzpicture}[x=1.0cm,y=1.0cm]
\draw (0,0) .. controls (3,-1) and (7,-1) .. (10,0); 
\draw (0,-2) .. controls (3,-3) and (7,-3) .. (10,-2); 
\fill [color=black] (5,-9) circle (1.0pt);
\fill [color=black] (5.5,-.75) circle (1.0pt);
\fill [color=black] (4.5,-2.75) circle (1.0pt);
\fill [color=black] (5.375,-6) circle (1.0pt);
\fill [color=black] (4.58,-6) circle (1.0pt);
\draw (5,-9) .. controls (5.5,-7) and (5.5,-2)  .. (5.5,-.75); 
\draw (5,-9) .. controls (4.5,-8) and (4.5,-3)  .. (4.5,-2.75); 
\draw (5,-6) ellipse (1cm and .25cm);
\draw[color=black] (5.2,-9) node {$p$};
\draw[color=black] (5.8,-1) node {$x_j$};
\draw[color=black] (4.2,-3) node {$x_i$};
\draw[color=black] (9,-.6) node {$X_j$};
\draw[color=black] (9,-2.6) node {$X_i$};
\draw[color=black] (6,-5.5) node {$x_j(t)$};
\draw[color=black] (4,-5.5) node {$x_i(t)$};
\draw[color=black] (6.4,-6) node {$\mathcal{C}_i^t$};
\end{tikzpicture}
\end{figure*}
The non-decreasing function $t\mapsto D_t$ may be bounded or not. In the first case, this implies  the sequence $(x_i)$ converges to a point $\xi\in\bo X$ that belongs to $\cap\bo X_i$. Since the projection on a closed convex subset is 1-Lipschitz, the point $\xi$ does not depend on $p$.  Thanks to the angular property of projection, for any $y\in X_i$, $\cangle_p(y,x_i)\leq\pi/2$ and so $\cap\bo X_i$ is included in the ball of radius $\pi/2$ around $\xi$.\\

Now, suppose  $t\mapsto D_t$ is not bounded. Let $n$ be the telescopic dimension of $X$. We choose $\delta>0$ such that $\delta\sqrt{2}+\sqrt{\frac{n}{(n+1)}}<1$. Let $D>0$ be a positive real given by \ref{cl1.3}. For $t\geq0$ such that $D_t>D$ and $i,j\in\N$ with $j\geq i\geq N_t$ we have $\cangle_{x_i}(p,x_j)\geq \pi/2$. So, $\cangle_{p}(x_i,x_j)\leq \pi/2$ and $d(x_i(t),x_j(t))\leq \sqrt{2}t$. Thanks to inequality (\ref{Jung})
\[r_i^t\leq\left(\delta+\sqrt{\frac{n}{2(n+1)}}\right)D_t\leq\left(\delta\sqrt{2}+\sqrt{\frac{n}{(n+1)}}\right)t.\]
So the triangle inequality gives
\[d(p,c_i^t)\geq d(p,x_i(t))-d(x_i(t),c_i^t)\geq \left[1-\left(\delta\sqrt{2}+\sqrt{\frac{n}{(n+1)}}\right)\right]t.\]
Inequality (\ref{cc}) shows $d(c_i^t,c_j^t)\leq\sqrt{2((r_i^t)^2-(r_j^t)^2)}$ and since $k\mapsto r_k^t$ is non-increasing for a fixed $t$ we deduce that $(c_k^t)_k$ is a Cauchy sequence. We denote by $c_t$ its limit and we remark that 
\begin{equation}\label{ct}
d(p,c_t)\geq\left[1-\left(\delta\sqrt{2}+\sqrt{\frac{n}{(n+1)}}\right)\right]t.
\end{equation}
Now we will  show $c_t$ converges to a point at infinity when $t$ goes to infinity. For $t'\geq t>0$ et $j\geq i\geq N_t'$ we introduce the point denoted by $\frac{t}{t'}c_i^{t'}$ on $[p,c_i^{t'}]$ at distance $\frac{t}{t'}d\left(p,c_i^{t'}\right)$ from $p$. Using Thales' Theorem in a comparison triangle, we have $d\left(x_j(t),\frac{t}{t'}c_i^{t'}\right)\leq \frac{t}{t'}d\left(x_j(t'),c_i^{t'}\right)\leq  \frac{t}{t'}r_i^{t'}$. And so, $r_i^t\leq \frac{t}{t'}r_i^{t'}$. Let $r_t$ be the (non-decreasing) limit of $(r_t^i)_i$ then $\frac{r^t}{t}\leq\frac{r^{t'}}{t'}\leq\sqrt{2}$. Thus $(\frac{r_t}{t})$ converges as $t$ goes to $+\infty$. In the geodesic triangle  $x_j(t),\frac{t}{t'}c_i^{t'},c_i^t$, if $m$ is the midpoint of $[\frac{t}{t'}c_i^{t'},c_i^t]$, the Bruhat-Tits inequality gives
\[d(x_j(t),m)^2\leq\frac{1}{2}\left(d\left(x_j(t),c_i^t\right)^2+d\left(x_j(t),\frac{t}{t'}c_i^{t'}\right)^2\right)-\frac{1}{4}d\left(\frac{t}{t'}c_i^{t'},c_i^t\right)^2\]
Since $m$ is not the circumcenter of $\mathcal{C}_i^t$, there exists $j$ such that $d(m,x_j(t))\geq r_i^t$. This last inequality with inequalities $d(x_j(t),c_i^t)\leq r_i^t$ and $d(x_j(t),\frac{t}{t'}c_i^{t'})\leq \frac{t}{t'}r_i^{t'}$ give $d\left(\frac{t}{t'}c_i^{t'},c_i^t\right)^2\leq 2t^2\left[\left(\frac{r_i^{t'}}{t'}\right)^2-\left(\frac{r_i^t}{t}\right)^2\right]$. If  $i$ goes to infinity we obtain 
\[d\left(\frac{t}{t'}c^{t'},c^t\right)^2\leq 2t^2\left[\left(\frac{r^{t'}}{t'}\right)^2-\left(\frac{r^t}{t}\right)^2\right]\]
Fix $t_0>0$ and $\epsilon>0$. To conclude, it suffices to show the points $x=\frac{t_0}{d(p,c^t)}c^t$ and $y=\frac{t_0}{d(p,c^{t})}\frac{t}{t'}c^{t'}$ on segments $[p,c^t]$ and  $[p,c^{t'}]$,  are at distance less than $\epsilon$ for $t,t'$ large enough. Once again, Thales' theorem in a comparison triangle associated with $p,c^t,\frac{t}{t'}c^{t'}$ shows $d(x,y)<\frac{t_0}{d(p,c^t)}d\left(\frac{t}{t'}c^{t'},c^t\right)$. Now, for $t,t'$ such that $\left[\left(\frac{r^{t'}}{t'}\right)^2-\left(\frac{r^t}{t}\right)^2\right]<\epsilon^2$, inequality (\ref{ct}) shows 
\[d(x,y)<\frac{\sqrt{2}\epsilon}{1-\left(\delta\sqrt{2}+\sqrt{\frac{n}{n+1}}\right)}.\]
Let $\xi$ be the limit of $c_t$. By the same argument as above, $\xi$ does not depend on $p$. If we choose $p\in X_i$ then convexity of $X_i$ shows that $\xi\in\partial X_i$. Finally, $\xi\in\cap_i \partial X_i$.\\

let $\eta\in\cap \partial X_i$ and $\rho$ be the geodesic ray from $p$ to $\eta$. Fix $i\in\N$ and denote by $p_i^u$ the projection of $\rho(u)$ on $X_i$. Since distance from $\rho(u)$ to $X_i$ is bounded, convexity of $x\mapsto d(x,X_i)$ implies that $d( p_i^u,\rho(u))$ is bounded by 
$d(p,X_i)$. So $p_i^u$ converges to $\eta$ as $u$ goes to infinity. Thus for $t>0$ and $i\geq N_t$, $\beta_\eta(x_i(t),p)=\lim_{u\to\infty}d(x_i(t),p_i^u)-d(p_i^u,p)$. Since $\angle_{x_i}(p,p_i^u)\geq\pi/2$ a comparison argument shows that $d(x_i(t),p_i^u)\leq d(p,p_i^u)$. Thus $\beta_\eta(x_i(t),p)\leq 0$. Continuity and convexity of $\beta_\eta$ imply $\beta_\eta(c^t,p)\leq 0$. Now, the asymptotic angle formula (\ref{aaf}) shows $\liminf_{t,u\to\infty}\cangle_x(\rho(u),c^t)\leq\pi/2$ and finally by inequality (\ref{sc}), $\angle(\xi,\eta)\leq\pi/2$.
\end{proof}
We call the point $\xi$ constructed in Proposition \ref{cl5.4} \emph{the center of directions} associated with the sequence $(X_i)$.
\begin{prop}\label{independant}Let $(X_i)$ and $(X_i')$ be two nested sequences of closed convex subsets of a complete CAT(0) space of  finite telescopic dimension. If for all $i,j\in\N$ there are $i',j'\in\N$ such that $X_{i}\subseteq X_{i'}'$ and $X'_{j}\subset X_{j'}$ then $(X_i)$ and $(X_i')$ have same centers of directions. 
\end{prop}
\begin{proof} If $\varphi\colon\N\to\N$ is an increasing map then $(X_i)$ and $(X_{\varphi(i)})$ haves same centers of directions. Indeed the curve $t\mapsto c^t$ in the proof of Proposition \ref{cl5.4} is the limit of the sequence $(c_i^t)_i$ and thus is the same for  $(X_i)$ and $(X_{\varphi(i)})$.\\

By inclusion properties of $(X_i)$ and $(X_i')$, we can find two extractions $\varphi$ and $\varphi'$ such that for all $i\in\N$, $X_{\varphi(i)}\subseteq X'_{\varphi'(i)}\subseteq X_{\varphi(i+1)}$. Then we set $X''_{2i}=X_{\phi(i)}$ and $X''_{2i+1}=X'_{\varphi'(i)}$. The above remark for $(X''_i)$ and $(X_i)$ and $(X''_i)$ and $(X'_i)$ concludes this proposition.

\end{proof}
There is an another place where gradient flows appear in \cite{MR2558883}. This is in the proof of our Proposition \ref{CL4.8}, which follows. The use of gradient flows is handy for the authors but really not necessary. We show how to modify the proof. Once it this done, all results of  \cite{MR2558883} may be obtain without the use of gradient flows.\\

Fix $x_0$ in a complete CAT(0) space $X$. Let $\mathcal{C}_0$ the set of all 1-Lipschitz and convex functions that vanish at $x_0$. Endowed with the pointwise convergence topology,  $\mathcal{C}_0$ is a compact topological space. Let $\mathcal{C}\subset\mathcal{C}_0$ the image of $X$ under the map $x\mapsto d(x,.)-d(x,x_0)$. 

\begin{prop}[Proposition 4.8 in \cite{MR2558883}]\label{CL4.8} Let $X$ a $\CAT(0)$ space of finite telescopic dimension not reduced to a point and with a minimal action of $\Isom(X)\action X$. Then every affine function of\ $\overline{\mathcal{C}}$ is a Busemann function associated with a point $\xi$  in the boundary of the de Rham factor of $X$. 
\end{prop}
\begin{lem}\label{1.2}Let $X$ be a complete CAT(0) space. If $f$ is an affine 1-Lipschitz map  such that
\begin{equation}\label{4.7eq}
\forall\epsilon>0\ \forall n\in\N\ \forall x\in X\ \exists z\in X,\ d(x,z)\geq n\ \mathrm{and}\ f(z)-f(x)\geq(1-\epsilon)d(x,z)
\end{equation}
then there exists $\xi\in\bo X$ such that for $x,y\in X$, 
\[f(x)=-\beta_{\xi}(x,x_0).\]
\end{lem}
\begin{proof} Fix $x\in X$. For $l>f(x)$, we set $X_l=\{y\in X|\ f(y)\geq l\}$. $X_l$ is a non-empty closed convex subset of $X$. Let $x_l$ be the projection of $x$ on $X_l$. For $f(x)<l'<l$ let $y$ be the unique point on $[x,x_l]$ such that $f(y)=l'$ then
\[d(x,x_{l'})\leq d(x,y)=\frac{l'-f(x)}{l-f(x)}d(x,x_l).\]
This shows that $l\mapsto \frac{l-f(x)}{d(x,x_l)}$ is non-increasing. Now hypothesis (\ref{4.7eq}) shows that $\lim_{l\to\infty} (l-f(x))/d(x,x_l)=1$. So $d(x,x_l)=l-f(x)$ for all $l>f(x)$ and $\cup_{l>f(x)}[x,x_l]$ is a geodesic ray. Let $\xi$ be the endpoint of this geodesic ray. Since the projection on $X_l$ is 1-Lipschitz $\xi$ does not depend on $x$. Now characterization of Busemann functions \cite[II.8.22]{MR1744486} shows that 
\[f(x)=-\beta_{\xi}(x,x_0).\]
\end{proof}

\begin{proof}[Proof of Proposition  \ref{CL4.8}]Let $f$ be an affine function of $\overline{\mathcal{C}}$. The proof of Proposition 4.8 in \cite{MR2558883} leads to the conclusion that $f$ and $-f$ satisfy the condition of lemma  \ref{1.2}. Thus there exists $\xi$ and $\xi'$ such that for all $x\in X$, $f(x)=-\beta_\xi(x,x_0)$ and $-f(x)=-\beta_{\xi'}(x,x_0)$. Since $f$ is 1-Lipschitz the concatenation of geodesic rays from $x$ to $\xi$ and $\xi'$ is a geodesic. This proves that $X$ is the reunion of geodesics with extremities $\xi$ and $\xi'$. Theorem II.2.14 in \cite{MR1744486} concludes the proof.
\end{proof}
\section{\texorpdfstring{Spherical and Euclidean buildings associated with $X_p(\K)$}{Spherical and Euclidean buildings}}
\subsection{Spherical building at infinity}\label{immeubles}
It is a classical result that the boundary at infinity (endowed with the Tits metric) of a Riemannian symmetric space of non-compact type is a spherical building (see section 3.6 of \cite{MR1441541}). We  show the same holds for $X_p(\K)$. We will use the following geometric definition of  a spherical building. It is borrowed from definition \cite[II.10A.1]{MR1744486} and this geometric definition is equivalent to the combinatorial usual one.
\begin{df}\label{defb} A \emph{spherical building} of dimension $n$ is a piecewise spherical simplicial complex $X$ such that :
\begin{enumerate}[(i)]
\item $X$ is the union of a collection $\mathcal{A}$ of subcomplexes $E$, called \emph{apartments}, such that the intrinsic metric $d_E$ on $E$ makes $(E,d_E)$ to the sphere $\mathbb{S}^n$ and induces the given spherical metric on each simplex. The $n$-simplices of a apartement are called \emph{chambers} and the (non-empty) intersections of chambers are called \emph{faces}.
\item Any two simplices of $X$ are contained in at least one apartment.
\item Given two apartments $E$ and $E'$ containing both simplices $B$ and $B'$, there exists a simplicial isometry from $(E,d_E)$ onto $(E',d_{E'})$ which leaves both $B$ and $B'$ pointwise fixed.
\end{enumerate}
If moreover, every $(n-1)$-simplex is a face of at least three $n$-simplices, $X$ is said to be a \emph{thick} building.
\end{df} 
In the case of the boundary $\bo X$ of a Riemannian symmetric space of non-compact type, apartments are exactly boundaries of maximal Euclidean subspaces of $X$.\\

If $X$ is a $\CAT(0)$ space, we recall that the \emph{Tits boundary} of $X$ is the space $\bo X$ endowed with the Tits metric (see definition \cite[II.9.18]{MR1744486} for more details).
\begin{prop} The Tits boundary of $X_p(\K)$ is a thick spherical building of dimension $p-1$.
\end{prop}
\begin{proof} We show that conditions (i)-(iii) of definition \ref{defb} hold. Apartments of $X_p(\K)$ are defined to be boundaries of maximal Euclidean subspaces of $X_p(\K)$. Thus, any apartment is isometric to $\mathbb{S}^{p-1}$. If $F$ is a maximal Euclidean subspace, a Chamber of $\bo F$ is the closure of a connected component of the set of points $\xi\in\bo F$ that are endpoint of regular geodesic ray included in $F$ (a geodesic ray is called regular if it is included in a unique maximal Euclidean space).  Conditions (i)-(iii) involve only finitely many apartments simultaneously. By Proposition \ref{fc}, we know such configurations lie actually in some $\bo Y$ where $Y$ is a totally geodesic subspace of $X_p(\K)$ isometric to some $X_{p,q}(\K)$ with $q\geq p$. Now, the building structure on $\bo X_{p,q}(\K)$ implies these conditions hold.
\end{proof}
\subsection{\texorpdfstring{Euclidean buildings as asymptotic cones}{Euclidean buildings}}
In \cite{MR1608566}, the authors introduce a new definition of Euclidean building, which is more geometric and more general than the usual one. This definition allows the authors to show that every asymptotic cone of a Euclidean building is a Euclidean building \cite[Corollary 5.1.3]{MR1608566}. Moreover, every asymptotic cone of Riemannian symmetric space of non-compact type is a Euclidean building. This phenomenon can be simply illustrated in dimension one. A Euclidean building (in the classical sense) of dimension one is a simplicial tree without leaf.  Since every asymptotic cone of a Gromov-hyperbolic space is a real tree \cite[example 2.B.(b)]{MR1253544}, some real trees that are not siplicial are also buildings of dimension 1 in the sense of Kleiner-Leeb.\\

We recall the definition of Euclidean building in the sense of Kleiner-Leeb. Let $E$ be a Euclidean space. Its boundary at infinity $\bo E$ endowed with the angular metric is a Euclidean sphere of dimension  dim$(E)-1$. Since isometries of $E$ are affine and translations act trivially on $\bo E$, one obtain a homomorphism 
\[\rho\colon\Isom(E)\to\Isom(\bo E)\]
that associates its linear part to every Euclidean isometry. A subgroup $W_\mathrm{Aff}\leq\Isom(E)$ is called an \emph{affine Weyl group} if it is generated by reflections through hyperplanes and if $W:=\rho(W_\mathrm{Aff})$ is a finite subgroup of $\Isom(\bo E)$. The group $W$ is called the \emph{spherical Weyl group} associated with $W_\mathrm{Aff}$. If $W_\mathrm{Aff}$ is an affine Weyl group then $(E, W_\mathrm{Aff})$ is called a \emph{Euclidean Coxeter complex} and $(\bo E,W)$ is the associated \emph{spherical Coxeter complex at infinity}. Its \emph{anisotropy polyhedron} is the spherical polyhedron 
\[\Delta:=\bo E/W.\]
An oriented segment (not reduced to a point) $\overline{xy}$ of $E$ determines  a unique point of $\bo E$ and the projection of this point to $\Delta$ is called the $\Delta$-\emph{direction} of  $\overline{xy}$. Let  $\pi$ be the projection $\bo E\to\Delta$. If $\delta_1,\delta_2$ are two points of $\Delta$, we introduce the finite set
\[D(\delta_1,\delta_2)=\{\angle(\xi_1,\xi_2))|\xi_1,\xi_2\in\bo E,\ \pi(\xi_1)=\delta_1,\ \pi(\xi_2)=\delta_2\}.\]
\begin{df}\label{dfb}Let $(E, W_\mathrm{Aff})$ be a Euclidean Coxeter complex. A \emph{Euclidean building modelled on} $(E, W_\mathrm{Aff})$ is a complete CAT(0) space $(X,d)$ with 
\begin{enumerate}[(i)]
\item a map $\theta$ from the set of oriented segments not reduced to a point to $\Delta$,
\item a collection, $\mathcal{A}$, called \emph{atlas}, of isometric embeddings $\iota\colon E\to X$ that preserve $\Delta$-directions. This atlas is closed under precomposition with isometries in $W_\mathrm{Aff}$. The image of such isometric embedding $\iota$ is called an \emph{apartment}.
\end{enumerate}
Moreover the following properties must hold.
\begin{enumerate}
\item For all $x,y,z\in X$ such that $y\neq z$ and $x\neq z$,
\[d_\Delta(\theta(\overline{xy}),\theta(\overline{xz}))\leq\overline{\angle}_x(y,z).\]
\item The angle between two geodesic segments  $\overline{xy}$ and $\overline{xz}$ is in $D(\theta(\overline{xy}),\theta(\overline{xz}))$.
\item Every geodesic segment, ray or line is contained in an apartment.
\item If $A_1$ and $A_2$ are two apartments with a non-empty intersection then the \emph{transition map} $\iota^{-1}_{A_2}\circ\iota_{A_1}\colon \iota_{A_1}^{-1}(A_1\cap A_2)\to\iota_{A_2}^{-1}(A_1\cap A_2)$ is the restriction of an element of $W_{\mathrm{Aff}}$.
\end{enumerate}
\end{df}
If $X$ is a Euclidean building, the \emph{rank} of $X$ is the dimension of any apartment.
\begin{proof}[Proof of Theorem \ref{eb}] Actually, the proof of Kleiner-Leeb, which shows asymptotic cones of Riemannian symmetric spaces of non-compact type are Euclidean buildings, works also in our infinite dimensional settings with a very slight modification. We recall how $\theta,\Delta,W_\mathrm{Aff}$ and apartments are defined and we refer to \cite[Theorem 5.2.1]{MR1608566} for the full proof and show the slight modification appears.\\

Choose a maximal Euclidean subspace $E$ in $X_p(\K)$ then $W_\mathrm{Aff}$ is defined to be the quotient group of the stabilizer Stab$(E)$ of $E$ by the pointwise stabilizer Fix$(E)$ of $E$. With the same notation as above we set $W=\rho(W_\mathrm{Aff})$.The quotient $\bo E/W$ can be identified with any fixed chamber $\Delta$ of the building at infinity. If $\xi\in\bo X_p(\K)$ we set $\theta(\xi)$ to be the unique point in $\Delta$ of the orbit of $\xi$ under Isom$(X_p(\K))$. This point exists because of Proposition \ref{fc} and the fact that Isom$(X_{p,q}(\K))$ acts transitively on chambers of $\bo  X_{p,q}(\K)$. Now, let $Y$ be an asymptotic cone of $X_p(\K)$. Apartments of $Y$ are defined to be ultralimits of maximal Euclidean subspaces of $X_p(\K)$ and if $x\neq y$ are points of $Y$, choose $(x_n)$ and $(y_n)$ sequences in $X_p(\K)$ corresponding respectively to $x$ and $y$. Let $\xi_n$ be the point at infinity of geodesic ray trough $y_n$ starting at $x_n$. Since $\Delta$ is compact, the sequence $(\theta(\xi_n))$ has a limit and $\theta(\overline{xy})$ is defined to be this limit. This does not depend on the choice of sequences because if $x,y,z$ are points of $X_p(\K)$ and $\xi,\eta$ are points at infinity corresponding to $\overline{xy}$ and $\overline{xz}$ then 
\[d_\Delta(\xi,\eta)\leq\cangle_x(y,z).\]

To show point (2) of definition \ref{dfb}, the authors use a compactness argument in \cite[Lemma 5.2.2]{MR1608566}. Let $x\in Y$ and $y,z\in Y\setminus\{p\}.$ Let $(x_n),(y_n),(z_n)$ be sequences of $X_p(\K)$ that correspond to respectively $x,y$ and $z$. Thanks to homogeneity and Proposition \ref{fc}, we can find a totally geodesic subspace $Z\subset X_p(\K)$ isometric to some $X_{p,2p}(\K)$ such that for any $n$ there exists $g_n$ isometry of $X_p(\K)$ such that $g_nx_n=E_0$ and $g_ny_n,g_nz_n\in Z$ for all $n$. Now, the argument given in \cite[Lemma 5.2.2]{MR1608566} works.

\end{proof}
\begin{proof}[Proof of corollary \ref{td}] We know that $X_p(\K)$ is a separable complete CAT(0) space. Thanks to Proposition \ref{fc}, every Euclidean subspace of $X_p(\K)$ is included in a convex subspace $Y$ which is isometric to some $X_{p,q}(\K)$. Since the rank of $X_{p,q}(\K)$ is $\min(p,q)$ (see table V in \cite[X.6]{MR1834454}), the rank of $X_p(\K)$ is less than $p$ and since there exist isometric embeddings of $X_{p,q}(\K)$ in $X_p(\K)$ with $q\geq p$, the rank of $X_p(\K)$ is $p$.\\

Thanks to Theorem \ref{eb}, every asymptotic cone of $X_p(\K)$ is a Euclidean building of dimension $p$. Now,  \cite[corollary 6.1.1]{MR1608566} asserts that if $V\subseteq U$ are open subsets of a Euclidean building $X$ then $H_k(U,V)=0$ for $k>\mathrm{rank}(X)$. This result and characterization \ref{car} show the geometric dimension of a building of rank $p$ is exactly $p$. So, the telescopic dimension of $X_p(\K)$ is $p$.
\end{proof}
\section{\texorpdfstring{Parabolic subgroups of $\O(p,\infty)$}{Parabolic subgroups of O(p,infinity)}}
It is a well-known fact that parabolic subgroups of SL$_n(\R)$ 	are in correspondence with flags of $\R^n$ (see \cite[2.17.27]{MR1441541} for example). A similar phenomenon is also true for $\O(p,\infty)$.\\

For the remaining of this section, $\K=\R$. A vector of $\mathcal{H}$ is \emph{isotropic} if $Q_p(x)=0$ and a subspace $E\subset \mathcal{H}$ is \emph{totally isotropic} if any $x\in E$ is isotropic. Since the index of $Q_p$ is $p$, any totally isotropic subspace has dimension less or equal to $p$. Maximal (for inclusion) totally isotropic subspaces are exactly those of dimension $p$. A sequence $(E_i)_{i=1}^k$ of non-trivial subspaces of $\mathcal{H}$ is called a \emph{flag} if $E_i\subset E_{i+1}$ for any $1\leq i\leq k-1$. A flag $F=(E_i)_{i=1}^k$ is said to be \emph{isotropic} if $E_k$ is a totally isotropic subspace of $\mathcal{H}$. We remark that $\O(p,\infty)$ acts naturally on the set of isotropic flags and this action gives an action of Isom$(X_p(\R))$.\\

We denote by $G_\xi$ the stabilizer of $\xi\in\bo X_p(\R)$ and by $G_F$ the stabilizer of an isotropic flag $F$, inside  Isom$(X_p(\R))$.\

\begin{prop}\label{parabolic}For any $\xi\in\bo X_p(\R)$ there exists an isotropic flag $F(\xi)$ such that $G_\xi=G_{F(\xi)}$. Moreover, for any totally isotropic flag $F$, there exists $\xi\in\bo X_p(\R)$ such that $F=F(\xi)$ and thus $G_F=G_\xi$.
\end{prop}

\begin{lem}\label{norm} Let $E\in X_p(\R)$, $\alpha_1$ be the first principal hyperbolic angle between $E$ and $E_0$. If $M\in\O(p,\infty)$ and $ME_0=E$ then $||M||=\sqrt{\cosh(\alpha_1)^2+\sinh(\alpha_1)^2}$.
\end{lem}

\begin{proof} The existence of biorthogonal bases show that, in a good orthogonal base of $\mathcal{H}$, we can find some $M_0\in\O(p,\infty)$ such that $M_0E_0=E$ and its matrix is 
\[\left[\begin{array}{ccc}
\cosh(\alpha)&\sinh(\alpha)&0 \\
\sinh(\alpha)&\cosh(\alpha)&0\\
0&0&\Id
\end{array}\right]\]
where $\cosh(\alpha)$ (respectively $\sinh(\alpha)$) is the matrix $\diag(\cosh(\alpha_1),\dots,\cosh(\alpha_p))$ (respectively the matrix $\diag(\sinh(\alpha_1),\dots,\sinh(\alpha_p))$. It is not difficult to show that $||M_0||=\sqrt{\cosh(\alpha_1)^2+\sinh(\alpha_1)^2}$. Now, if $M\in\O(p,\infty)$ satisfies $ME_0=E$ then $MM_0^{-1}\in\O(p)\times\O(\infty)$ and thus $||M||=||M_0||$. 
\end{proof}

\begin{proof}[Proof of Proposition \ref{parabolic}] Let $g_t$ be the transvection of length $t$ from $E_0$ toward $\xi$. Let $h$ be an isometry of $X_p(\R)$ then $h\xi=\xi$ if and only if the geodesic ray from $E_0$ to $\xi$ and its image by $h$ remain at a bounded distance one from another. This means exactly the set $\{d(hg_tE_0,g_tE_0)|\ t\geq0\}=\{d(g_{t}^{-1}hg_tE_0,E_0)|\ t\geq0\}$ is bounded. Thanks to Lemma \ref{norm}, this means that the set of operators $\{g_{t}^{-1}hg_t|\ t\geq0\}$ is bounded.\\

Since the isometry group of $X_{p,q}(\R)$ acts transitively on the set of chambers of the spherical building $\bo X_{p,q}(\R)$ and any $\xi\in\bo X_p(\R)$ is in the closure of a chamber, we can suppose that $\xi$ is in the closure of the boundary at infinity of the Weyl chamber (of second type in the terminology of \cite[2.12.4]{MR1441541}) $\mathcal{C}=\{\exp(H_\lambda)E_0|\ \lambda_1>\dots>\lambda_p>0\}$ where $H_\lambda$ is an infinite-dimensional operator of finite rank of $\mathcal{H}$ which has the same expression as the one in equation (\ref{eqhl}).  So, we suppose that $\xi$ is the limit when $t\to+\infty$ of $g_tE_0=\exp(tH_\lambda)E_0$ for a fixed $\lambda=(\lambda_1,\dots,\lambda_p)$ with $\lambda_1\geq\dots\geq\lambda_p\geq0$. Let $v_1>\dots>v_k$ the distinct non-trivial values of $\lambda_1,\dots,\lambda_p$ and $E_i$ be the span of $\{e_j+e_{p+j}|\ \lambda_j\geq v_i\}$. In order to show that $h$ stabilizes the isotropic flag $F=(E_i)_{i=1}^k$, we will use a more convenient Hilbert base of $\mathcal{H}$. Let $(e'_i)$ the Hilbert base defined by 

\[\left\{\begin{array}{ll}
e'_i=1/\sqrt{2}(e_i+e_{i+p}),\ & 1\leq i\leq p\\
e'_i=1/\sqrt{2}(e_i-e_{i+p}),\ & p+1\leq i\leq 2p\\
e'_i=e_i,\ & i>2p
\end{array}\right.
\]
In this new base, the block decomposition of the matrix representation of $g_t$ is 
\[\left[\begin{array}{ccc}
e^{t\lambda}&0&0\\
0&e^{-t\lambda}&0\\
0&0&I
\end{array}\right]\]
where $e^{t\lambda}$ is the diagonal matrix with diagonal $\exp(t\lambda_1),\dots,\exp(t\lambda_p)$. The matrix of  $\Phi$ is 
\[\left[\begin{array}{ccc}
0&I_p&0\\
I_p&0&0\\
0&0&-I
\end{array}\right].\]
If we write the matrix of $h$
\[\left[\begin{array}{ccc}
h_1&h_2&h_3\\
h_4&h_5&h_6\\
h_7&h_8&h_9
\end{array}\right]\]
then the matrix of $g^{-1}_thg_t$ is
\[\left[\begin{array}{ccc}
e^{-\lambda t}h_1e^{\lambda t}&e^{-\lambda t}h_2e^{-\lambda t}&e^{-\lambda t}h_3\\
e^{\lambda t}h_4e^{\lambda t}&e^{\lambda t}h_5e^{-\lambda t}&e^{\lambda t}h_6\\
h_7e^{\lambda t}&h_8e^{-\lambda t}&h_9
\end{array}\right].\]

Now, since $\{||g^{-1}_thg_t||\,|\ t\geq0\}$ is bounded, simple computations show that $h_1$ is a block upper-triangular matrix and blocks correspond with $E_i$'s. The matrix $h_4$ has zeros everywhere except if the row index, $i$, and column index, $j$, satisfy $\lambda_i=\lambda_j=0$. The matrix $h_7$ has trivial columns except the ones whose index, $j$, satisfies $\lambda_j=0$. This shows $h$ stabilizes the flag $F$.\\

Conversely, if $F$ is an isotropic flag $(E_i)_{i=1}^k$ then we can find a Hilbert base $(e_i)_{i\in\N}$ such that there exist $1=i_1<\dots<i_{k+1}\leq p+1$ such that $E_j$ is the span of $e_{i_j},\dots,e_{i_{j+1}-1}$. We define $\lambda_i=k-j+1$ if $i_j\leq i<i_{j+1}$ and $\lambda_i=0$ if $i_{k+1}\leq i\leq p$. Now, if $h\in G_F$ then we use the same matrix representation (with same block decomposition as above) in the base $(e'_i)$ constructed as above from $(e_i)$. The block $h_1$ is block upper-triangular and $h_4,h_7$ have trivial columns except the ones whose index, $j$, satisfies $\lambda_j=0$.\\ 

We also know that $h^{-1}\in G_F$ and $^th=\Phi h^{-1}\Phi$. If $h'_i$ are the blocks of $h^{-1}$ then $h_5=^th'_1$ is block lower-triangular, $^th_4=h'_4$ have trivial entries except the ones whose index, $i,j$, satisfy $\lambda_i=\lambda_j=0$ and $h_6=-^th'_7$ has trivial rows except the ones whose index, $i$, satisfies $\lambda_i=0$. These conditions on the blocks of $h$ imply that $\{||g^{-1}_thg_t||\,|\ t\geq0\}$ is bounded and if $\xi=\lim_{t\to\infty}g_t E_0$ then $h\in G_\xi$.
\end{proof}
\part{Furstenberg Maps}
\section{Amenability}
\subsection{Amenable actions}
We recall the notion of amenable actions, which generalizes the notion of amenable groups and was introduced by R. Zimmer in \cite{MR0473096}. See \cite[section 4]{MR776417} and \cite{MR1799683} for more details.
 Let $\Omega$ be a standard Borel space and $G$ be a locally compact second countable group. The space $\Omega$ is said to be a $G$-space if it is endowed with an action $G\action\Omega$ by Borel automorphisms and there is a quasi-invariant probability Borel measure  on $\Omega$. Every measurable notion on $B$ will refer to this implicit class of measure. \\

Throughout this section $\Omega$ will be a standard Borel space and $\mu$  a Borel measure on it.
\begin{df} A measurable field of Banach spaces is collection $\E=\{(E_\omega,||\||)\}_{\omega\in\Omega}$ of Banach spaces and a subset $\mathcal{M}\subset\prod_{\omega\in\Omega}E_\omega$ with the following properties :
\begin{enumerate}[(i)]
\item if $f,g\in\mathcal{M}$ then $f+g\in\mathcal{M}$,
\item if $f\in\mathcal{M}$ and $\phi\colon\Omega\to\C$ is a measurable function then $\phi f\in\mathcal{M}$,  
\item if $f\in\mathcal{M}$ then $\omega\to||f_\omega||$ is measurable,
\item if $f\in\prod_{\omega\in\Omega}E_\omega$ such that $(f^n)$ is a sequence in $\mathcal{L}$ with $\lim f^n_\omega=f_\omega$ for almost every $\omega$ then $f\in\mathcal{M}$ and
\item For almost every $\omega$, $\{f_\omega|\ f\in\mathcal{M}\}$ is dense in $E_\omega$.
\end{enumerate}
The subset $\mathcal{M}$ is called a \emph{measurable structure} for $\E$ and elements of $\mathcal{M}$ are called sections of $\E$. The measurable field $\E$ is \emph{separable} if there is a countable family $\{f^n\}\in\mathcal{M}$ such that $(f^n_\omega)_{n\in\N}$ is dense in $E_\omega$ for almost every $\omega$.
\end{df}
If $\E$ is a separable measurable field of Banach spaces, a \emph{cocycle} $\alpha$,for $G$ on $\E$, is a collection $\{\alpha(g,\omega)\}_{g\in G,\omega\in\Omega}$ such that
\begin{enumerate}[(i)]
\item for all $g$ and almost every $\omega$, $\alpha(g,\omega)\in$Isom$(E_\omega,E_{g\omega})$,
\item for all $g,g$ and almost every $\omega$, $\alpha(gg',\omega)=\alpha(g,g'\omega)\alpha(g',\omega)$ and
\item for all $f,f'\in\mathcal{M}$, $(g,\omega)\mapsto||f_\omega-\alpha(g,g^{-1}\omega)f'_{g^{-1}\omega}||$ is measurable.
\end{enumerate} 
In this case the formula $(g f)_\omega=\alpha(g,g^{-1}\omega)f'_{g^{-1}\omega}$ defines an action of $G$ on $\mathcal{M}$. If $\E$ is measurable field endowed with a cocycle for $G$ then one can constructs the dual field $\E^*$ endowed with the dual cocycle $\alpha^\sharp$.
\begin{df}Let $\Omega$ be a $G$-space. The action $G\action\Omega$ is amenable if for every cocycle for $G$ on a measurable field $\E$  over $\Omega$ and every $G$-invariant subfield $\mathbf{K}$ of weakly compact subsets of the balls of $\E^*$ there exists an invariant section in $\mathbf{K}$.
\end{df}
\subsection{G-boundaries}\label{G-boundary} We recall the notion of $G$-boundary, which appeared for the first time  in \cite{MR1325797}.
\begin{df} Let G be a locally compact group and $(B,\nu)$ a $G$-space. The measure space $(B,\nu)$ is said to be a $G$-boundary if
\begin{enumerate}[(i)]
\item the action $G\action(B,\nu)$ is amenable,
\item the diagonal action $G\action(B\times B,\nu\times\nu)$ is ergodic.
\end{enumerate}
\end{df}
Thanks to a theorem of V. Kaimanovich in \cite{MR2006560} (which generalizes \cite[Theorem 6]{MR1911660}), every locally compact and second countable group has a strong boundary, which a strengthening of the notion of boundary  and which has been introduced  by M. Burger and N. Monod in \cite{MR1911660}.
\section{Measurable fields of CAT(0) spaces}\label{mfcs}
A general study of measurable fields of CAT(0) spaces has been done in \cite{Anderegg:2011fk}. We first recall definitions and some general lemmas which are part of this general study.
\begin{df}Let $(\Omega,\mu)$ be a standard probability space. A \emph{measurable field of  CAT(0) spaces} is a collection $\X=\{(X_\omega,d_\omega)\}$ of (non-empty) complete CAT(0) spaces and a countable family $\mathcal{F}\subset\prod_\omega X_\omega$, called a \emph{fundamental family}, such that
\begin{enumerate}[(i)]
\item for all $x,y\in\mathcal{F}$, $\omega\mapsto d_\omega(x_\omega,y_\omega)$ is measurable,
\item for almost all $\omega$, $\{f_\omega|\ f\in\mathcal{F}\}$ is dense in $X_\omega$.
\end{enumerate}
\end{df}
Let $\X$ be a measurable field of CAT(0) spaces. A \emph{section} of $\X$ is an element $x\in\prod_\omega X_\omega$ such that for all $y\in\mathcal{F}$, $\omega\mapsto d_\omega(x_\omega,y_\omega)$ is measurable. Two sections are identified if they agree almost everywhere. The set of all sections is the \emph{measurable structure} $\mathcal{M}$ of $\X$. If $x,y$ are two sections, the equality
\[d_\omega(x_\omega,y_\omega)=\sup_{z\in\mathcal{F}}|d_\omega(x_\omega,z_\omega)-d_\omega(z_\omega,y_\omega)|\]
shows that $\omega\mapsto d_\omega(x_\omega,y_\omega)$ is also measurable. Since a pointwise limit of measurable maps is also measurable, we have the following lemma.
\begin{lem}If $x\in\prod_\omega X_\omega$ and $(x^n)$ a sequence of sections such that for almost every $\omega$, $x^n_\omega\to x_\omega$ then $x$ is a section of $\X$. 
\end{lem}
%
%
If $G$ is a locally compact group and $\Omega$ is a $G$-space then a \emph{cocycle} for $G$ on $\X$ is a collection $\{\alpha(g,\omega)\}_{g\in G,\omega\in\Omega}$ such that 
\begin{enumerate}[(i)]
\item for all $g$ and almost every $\omega$, $\alpha(g,\omega)\in$ Isom$(X_\omega,X_{g\omega})$,
\item for all $g,g$ and almost every $\omega$, $\alpha(gg',\omega)=\alpha(g,g'\omega)\alpha(g',\omega)$ and
\item for all $x,y\in\mathcal{F}$, $(g,\omega)\mapsto d_\omega(x_\omega,\alpha(g,g^{-1}\omega)y_{g^{-1}\omega})$ is measurable.
\end{enumerate}
A \emph{subfield} $\mathbf{Y}$ of $\X$ is a collection $\{Y_\omega\}_{\omega\in\Omega}$ of  non-empty closed convex subset such that for every section $x$ of $\X$, the function $\omega\mapsto d(x_\omega,Y_\omega)$ is measurable.\\

We identify subfields $\mathbf{Y}$ and $\mathbf{Y}'$ if $Y_\omega=Y_\omega'$ for almost every $\omega$. We introduce a partial order on the set of (equivalence class of) subfields : $\mathbf{Y}\leq\mathbf{Y}'$ if for almost every $\omega$, $Y_\omega\subseteq Y_\omega'$.\\

A cocycle for $G$ on $\X$ induces an action  of $G$ on $\mathcal{M}$ by $(gx)_\omega=\alpha(g,g^{-1}\omega)x_{g^{-1}\omega}$ for $x\in\mathcal{M}$ and $\omega\in\Omega$. It induces also an action on subfields  by $g\mathbf{Y}=\{\alpha(g,g^{-1}\omega)Y_{g^{-1}\omega}\}_{\omega}$.
\begin{lem}\label{order} Let $\mathbf{Y},\mathbf{Z}$ be two subfields of $\X$. Then $\mathbf{Y}\leq\mathbf{Z}$ if and only if for all $x\in\mathcal{F}$ and almost every $\omega$, $d_\omega(Y_\omega,x_\omega)\geq d_\omega(Z_\omega,x_\omega)$ and $\mathbf{Y}<\mathbf{Z}$ if and only if $\mathbf{Y}\leq\mathbf{Z}$ and there exists $x\in\mathcal{F}$ such that $\mu(\{\omega|\ d_\omega(Y_\omega,x_\omega)> d_\omega(Z_\omega,x_\omega)\})>0$.
\end{lem}
\begin{proof}For $x\in\mathcal{F}$, we define $\Omega_x=\{\omega\ |\ d_\omega(x_\omega,Y_\omega)> d_\omega(x_\omega,Z_\omega)\}$ and $\overline{\Omega}_x=\{\omega\ |\ d_\omega(x_\omega,Y_\omega)\geq d_\omega(x_\omega,Z_\omega)\}$. Then $\Omega_x$ et $\overline{\Omega}_x$ are measurable subsets of $\Omega$. If $\mathbf{Y}\leq\mathbf{Z}$ then for all $x\in\mathcal{F}$, $\overline{\Omega}_x$ has full measure. Conversely, if for all $x\in\mathcal{F}$, $\overline{\Omega}_x$ has full measure , then $\overline{\Omega}=\cap_x\ \overline{\Omega}_x$ has also full measure  and for all $\omega\in\overline{\Omega}$, $Y_\omega\subseteq Z_\omega$. Moreover, if $\mu(\Omega_x)>0$ then $\mathbf{Z}\nleq\mathbf{Y}$. This shows the second part of the lemma. 
\end{proof}
Following lemmas aim to show that usual constructions in complete CAT(0) spaces can be done measurably for measurable fields of CAT(0) spaces.
\begin{lem}If $x,y$ are two sections and $r\colon\Omega\to[0,+\infty)$ is a measurable map then $\omega\mapsto d_\omega(y_\omega,B(x_\omega,r(\omega)))$ is measurable.
\end{lem}
\begin{proof}The restriction of $\omega\mapsto d_\omega(y_\omega,B(x_\omega,r(\omega)))$ to $\{\omega|\ r(\omega)=0\}$ is clearly measurable. So we suppose that $r(\omega)>0$ for all $\omega$. Then we remark that
\[d_\omega(y_\omega,B(x_\omega,r(\omega)))=\inf\{d_\omega(y_\omega,z_\omega)|\ z\in\mathcal{F}, z_\omega\in B(x_\omega,r(\omega))\}.\]
For $z\in\mathcal{F}$, we define $\overline{d}_\omega(y_\omega,z_\omega)=+\infty$ if $z_\omega\notin B(x_\omega,r(\omega))$ and $\overline{d}_\omega(y_\omega,z_\omega)=d_\omega(y_\omega,z_\omega)$ in the other case. Then $\omega\mapsto \overline{d}_\omega(y_\omega,z_\omega)$ is a measurable function with values in $\R^+\cup\{\infty\}$ and
\[d_\omega(y_\omega,B(x_\omega,r(\omega)))=\inf_{z\in\mathcal{F}} \overline{d}_\omega(y_\omega,z_\omega).\]
\end{proof}
\begin{lem}\label{1} Let $\mathbf{Y}$ be a subfield of $\X$. If $x$ is a section of $\X$ then the family of projections of $x_\omega$ on $Y_\omega$ is a section of $\mathbf{X}$. 
\end{lem}
\begin{proof}
Let $\mathcal{F}=\{x^i\}_{i\in\N}$ be a fundamental family of $\X$ and $x$ be a section of $\X$. We define
\[i(n,\omega)=\inf\left\{i\ \left|\ \begin{array}{ccl}d(x_\omega,x^i_\omega)&\leq& d(x_\omega,Y_\omega)+1/n\ \mathrm{and}\\
d(x^i_\omega,Y_\omega)&\leq&1/n\end{array}\right.\right\}.\]
Thus, $(x^{i(n,\omega)}_\omega)$ is a section and $x^{i(n,\omega)}_\omega\longrightarrow \pi_{Y_\omega}(x_\omega)$.
\end{proof}
\begin{rem} Lemma \ref{1} shows subfields are fields on their own.  A fundamental family is given by projections of elements of a fundamental family of $\X$. 
\end{rem}
\begin{lem}\label{2}Let $\{x^i\}_{i\in\N}$ be a countable family of sections of $\X$. The function $\omega\mapsto$rad$(\{x^i_\omega\})$ is measurable and if it is essentially bounded then the family of circumcenters of $\{x^i_\omega\}$ , is a section.
\end{lem}
\begin{proof}The first part of the lemma holds because
\[\mathrm{rad}(\{x^i_\omega\})=\inf_{y\in\mathcal{F}}\sup_{i\in\N}d_\omega(x^i_\omega,y_\omega).\]
Let $r(\omega)=$rad$(\{x_\omega^i\})$ and let us number $\mathcal{F}=\{y^j\}_{j\in\N}$. We define
 \[j(n,\omega)=\min\{j\in\N|\ \sup_{i\in\N} d_\omega(x_\omega^i,y^j_\omega)\leq r(\omega)+1/n\}.\]
 Then $(y_\omega^{j(n,\omega)})$ is a section and for almost every $\omega$,
 $\lim_{n\to+\infty} y_\omega^{j(n,\omega)}$ is the circumcenter of $\{x_\omega^i\}$.
\end{proof}
\begin{lem}\label{3} Let $x,y$ be two sections of $\X$ and $d\colon\Omega\to[0,+\infty)$ a be measurable function such that for almost every $\omega$, $d(\omega)\leq d_\omega(x_\omega,y_\omega)$. The family $(z_\omega)$ of points on $[x_\omega,y_\omega]$ such that $d_\omega(x_\omega,z_\omega)=d(\omega)$, is a section of $\X$.
\end{lem}
\begin{proof} Any such function $d$ can be obtain as a pointwise limit of function of the type $\omega\mapsto\lambda(\omega) d_\omega(x_\omega,y_\omega)$ where $\lambda$ is a measurable function with dyadic values in $[0,1]$. Thus, it suffices to show the result when $z_\omega$ is the midpoint of $[x_\omega,y_\omega]$.\\
In a $\CAT(0)$ space, $X$, fix two points $x$ and $y$ then the set
\[Z_\varepsilon=\left\{z\in X|\ \max(d(x,z),d(y,z))\leq\frac{d(x,y)+\varepsilon}{2}\right\}\]
contains the midpoint of $[x,y]$ and has diameter at most $\varepsilon$. It suffices to define $z_\omega$ to be the limit as $n\to+\infty$ of projections of $x_\omega$ on the intersection   
\[\overline{B}\left(x_\omega,\frac{d(x_\omega,y_\omega)+1/n}{2}\right)\cap\overline{B}\left(y_\omega,\frac{d(x_\omega,y_\omega)+1/n}{2}\right).\]
\end{proof}
A measurable field of CAT(0) spaces $\X$ has finite telescopic dimension if for almost every $\omega$, $X_\omega$ has finite telescopic dimension. We note that the quantitative result of Theorem \ref{cl1.3} shows that $\omega\mapsto$ DimTel$(X_\omega)$ is a measurable map. For example if there is a cocycle for $G$ on $\X$ and $G\action \Omega$ is ergodic then almost every $X_\omega$ has the same telescopic dimension (maybe infinite).\\

Let $\X$ be a measurable field of CAT(0) spaces such that for almost $\omega$, $\bo X_\omega\neq\emptyset$. We define its \emph{boundary field} $\bo\X$ to be the collection $(\bo X_\omega)$. A section of $\bo\X$ is a collection $\xi=(\xi_\omega)$ such that for all $x,y$ sections of $\X$, the function 
\[\omega\mapsto\beta_{\xi_\omega}(x_\omega,y_\omega)\]
is measurable. 
\begin{lem} Let $\X$ be a measurable field of $\CAT(0)$ spaces with almost surely non-empty boundary. Let $\xi=(\xi_\omega)$ be a collection of points $\xi_\omega\in\bo X_\omega$.\\
The collection $\xi$ is a section of $\bo \X$ if and only if there exists a sequence $(z^n)$ of sections of $\X$ such that  for almost every $\omega$, $z_\omega^n\longrightarrow\xi_\omega$.
\end{lem}
\begin{proof}Let $(z^n)$ be such a sequence of sections. Thus for all sections $x,y$ of $\X$ and almost every $\omega$, 
\[\beta_{\xi_\omega}(x_\omega,y_\omega)=\lim_{n\to+\infty}d_\omega(x_\omega,z^n_\omega)-d_\omega(y_\omega,z^n_\omega).\]
Conversely, let $\xi$ be a section of $\bo\X$. We fix a section $x$ of $\X$ and we define $z^n_\omega$ to be the point on the geodesic ray from $x_\omega$ to $\xi_\omega$ at distance $n$ from $x_\omega$. Let $Y^n_\omega=\{y\in X_\omega|\ \beta_{\xi_\omega}(y,x_n)\leq-n\}$. If $z$ is a section of $\X$ then 
\[d(z_\omega,Y_\omega)=\max\left\{0,\beta_{\xi_\omega}(z_\omega,x_\omega)+n\right\}.\]
Thus $\mathbf{Y^n}=\{Y^n_\omega\}$ is a subfield of $\X$ and the collection $z^n=(z^n_\omega)=\left(\pi_{Y^n_\omega}(x_\omega)\right)$ is a section of $\X$ which tends to $\xi$.
\end{proof}
If $\Omega$ is a $G$-space for some locally compact group $G$ and $\alpha$ is a cocycle for $G$ on $\X$ then there exists a natural action of $G$ on  the sections of $\bo\X$. This is given by $(g\xi)_\omega=\alpha(g,g^{-1}\omega)\xi_{g^{-1}\omega}$.
\begin{prop}\label{bsec} Suppose that $\X$ has finite telescopic dimension and $(\X^n)$ is a non-increasing sequence of subfields such that for almost every $\omega$, $\cap_n X^n_\omega=\emptyset$. Let $\xi_\omega$ be center of directions constructed in Proposition \ref{cl5.4} associated with $(X^n_\omega)$. \\

Then $\xi=(\xi_\omega)$ is a section of $\bo \X$.
\end{prop}
%
%
\begin{proof}Proof of Proposition \ref{cl5.4} and lemmas \ref{1}, \ref{2} and \ref{3} show that for almost every $\omega$, $\xi_\omega$ is a limit of a sequence $z^n_\omega$ where $z^n$ is a section of $\X$. 
\end{proof}
\begin{prop}\label{ibs}Suppose $\X$ has finite telescopic, $G$ acts on $\X$ via a cocycle $\alpha$ and $G\action\Omega$ is ergodic. Then there exists a minimal invariant subfield of $\X$ or there exists an invariant section of $\bo\X$.
\end{prop}
The following lemma will be usefull for the  proof of Proposition \ref{ibs}.
\begin{lem}\label{cofinal} Let  $\mathcal{X}$ be a totally ordered family of subfields of $\X$. Then there exists a countable non-increasing  subfamily $(\X^n)_{n\in\N}$ that is cofinal.
\end{lem}
We recall that $(\X^n)_{n\in\N}$ is cofinal means that for any $\mathbf{Y}\in\mathcal{X}$ there is $n\in \N$ such that $\X^n\leq\mathbf{Y}$.
\begin{proof}For $x\in\mathcal{F}$ and $\mathbf{Y}\in\mathcal{X}$, set $f_x^\mathbf{Y}(\omega)=d_\omega(x_\omega,\mathbf{Y}_\omega)$. Then for all $x\in\mathcal{F}$ and $\mathbf{Y}\in\mathcal{X}$, $f_x^\mathbf{X}$ is a measurable function and (Lemma \ref{order}) \[\mathbf{Y}\geq\mathbf{Z}\ \iff\ \forall x\ \mathring{\forall} \omega\ f_x^\mathbf{Y}(\omega)\geq f_x^\mathbf{Z}(\omega).\]
 
Now thanks to a classical analysis result,  for all $x\in\mathcal{F}$, we can find a sequence $(\X^n)$ such that $(f_x^{\X^n})_n$ is non-inreasing and cofinal among $\{f^\mathbf{Y}_x\}_{\mathbf{Y}\in\mathcal{X}}$ (for the order : $f\geq g\ \iff\ \mathring{\forall} \omega\ f(\omega)\geq g(\omega)$). Since $\mathcal{F}$ is countable, we can suppose this is the case for all $x\in\mathcal{F}$ simultaneously. Lemma \ref{order} permits to conclude that  $(\X^n)$ is cofinal.

\end{proof}
\begin{proof}[Proof of Proposition \ref{ibs}]We suppose there is no invariant section of $\bo\X$ then we will show that the set of all  invariant subfields of $\X$ is inductive (for the opposite order of $\geq$). Then Zorn's Lemma will provide a minimal invariant subfield.\\

Let $\mathcal{X}$ be a totally ordered subset of invariant subfields. Thanks to Lemma \ref{cofinal}, $\mathcal{X}$ contains a cofinal non-increasing sequence $(\X^n)$. The subset $\{\omega\in\Omega|\  \bigcap_{n} X_\omega^n=\emptyset\}=\{\omega\in\Omega|\  d_\omega(x_\omega,X_\omega^n)\to+\infty\}$ is measurable and $G$-invariant. By ergodicity, it is a null or a conull set. If it is a null set then Proposition \ref{bsec} provides a section at infinity which is invariant because $\X_n$ are invariant. This contradicts our assumption. In the other case we define $Y_\omega=\bigcap_n X_\omega^n$. This is a closed convex subset of $X_\omega$ for almost every $\omega$. If $\mathcal{F}=\{x^k\}_{k\in\N}$ then we set $y^k_\omega$ to be the projection of $x_\omega^k$ on $Y_\omega$. Since $y^k_\omega$ is the limit of projections of $x_\omega^k$ on $X_\omega^n$ for almost every $\omega$, $y^k=(y_\omega^k)$ is a section of $\X$ and $\{y^k\}$ is a fundamental family for $\mathbf{Y}$ which  is thus an invariant subfield and a lower bound for $\mathcal{X}$.
\end{proof}

\section{\texorpdfstring{Adams-Ballmann Theorem in a measurable context}{Adams-Ballmann theorem}}

\subsection{Euclidean de Rham  factor}Let $X$ be a complete CAT(0) space. It is a classical result that $X$ is isometric  to some product $Y\times E$ where $Y$ is a complete CAT(0) space and $E$ is Euclidean space (maybe of infinite dimension) that is maximal for inclusion. Moreover if $X\simeq Y'\times X'$ is another such decomposition then $E\simeq E'$ and $Y\simeq Y'$. The Euclidean space $E$ is called the Euclidean de Rham factor of $X$. We show that under some assumptions, this decomposition can be done measurably. Under assumptions of properness and finite dimension, M. Anderegg and P. Henry show that a more precise decomposition, inspired by \cite[Theorem 1.6]{MR1645958}, can be obtained (see \cite[Proposition 3.20]{Anderegg:2011fk}).\\
 
Throughout this section we consider a measurable field of CAT(0) spaces $\X$ that has finite telescopic dimension and is not reduce to a single section, a fundamental family $\mathcal{F}=\{x^i\}$ for $\X$ and a cocycle $\alpha$ for  a locally compact group second countable group $G$ on $\X$. We assume that $\X$ is minimal and $G\action \Omega$ is ergodic.
\begin{rem}Assumptions of ergodicity and minimality imply that for almost every $\omega$, $X_\omega$ is unbounded.
\end{rem}
\begin{prop}\label{Rham}Let $x$ be a section of $\X$. There exists $n\in\N$ and two subfields $\mathbf{E}$ and $\mathbf{Y}$ of $\X$ containing $x$ such that $\X=\mathbf{E}\times\mathbf{Y}$, for almost every $\omega$, $E_\omega\simeq\R^n$ and $\mathbf{E}$ is maximal for those properties.\\

Moreover, if $y$ is an other section of $\X$ and $\X=\mathbf{E}'\times\mathbf{Y}'$ is another such decomposition associated with $y$ then for almost every $\omega$, the projections $\pi_{E_\omega}|_{E_\omega'}$ and  $\pi_{Y_\omega}|_{Y_\omega'}$ are isometries. In particular, if $x=y$ then $\mathbf{E}=\mathbf{E}'$ and $\mathbf{Y}=\mathbf{Y}'$.
\end{prop}
The subfield $\mathbf{E}$ will be called the \emph{Euclidean de Rham factor} of $\X$ and such a decomposition $\X=\mathbf{E}\times\mathbf{Y}$ will be called a \emph{Euclidean de Rham decomposition}. We will say that $\X$ is Euclidean if $\X=\mathbf{E}$ and that $\X$ has no Euclidean factor is $\X=\mathbf{Y}$ is reduced to a point. \\

To recover measurably the Euclidean de Rham Factor we well need a measurable version of Proposition \ref{CL4.8}. We fix a section $x^0$ of $\X$. We set $\mathcal{A}$ to be set of family $f=(f_\omega)$ such that such that for almost $\omega$, $f_\omega$ is an affine function on $X_\omega$ and there exist a sequence $(x^ i)$ of sections of $\X$ such that for almost every $\omega$ and every $y\in X_\omega$,
 \[f_\omega(y)=\lim_{i\to\infty}d_\omega(y,x^i_\omega)-d(x_\omega^0,x_\omega^i).\]
For all $\omega\in\Omega$ we set $E_\omega$ the Euclidean subspace of $X_\omega$ that contains $x_\omega^0$ and is isometric to the Euclidean de Rham factor of $X_\omega$. At this stage, we don't that $(E_\omega)$ is a subfield.
\begin{prop}\label{mesu}Let $f\in\mathcal{A}$.Then for almost every  $\omega$, there exists $\xi_\omega\in\bo E_\omega$ such that  $f_\omega$ is the Busemann function associated with $\xi_\omega$.  
 \end{prop}
Proposition \ref{CL4.8} uses the following technical lemma.
 \begin{lem}[Lemma 4.9 in  \cite{MR2558883}]\label{cms} Let $X$ be an unbounded complete CAT(0) space of finite telescopic dimension . Then there exists a sequence $(D_j)$ of positive numbers such that for every  $j$, $D_j>j$, a sequence $(\delta_j)$ of positive numbers that tends to $0$, a sequence of points $p_j\in X$ and a sequence of finite subsets $Q_j\subset X$ with the following properties.
 \begin{enumerate}[(i)]
 \item The set $Q_j$ is included in the closed ball of radius $D_j(1+\delta_j)$ centered at $p_j$.
 \item For every $s\in X$, there exists $q_j\in Q_j$ such that $d(s,q_j)-d(s,p_j)\geq D_j-1$.
 \end{enumerate}
 \end{lem}
 A measurable version for $\X$ of this lemma is the following.
 \begin{lem}\label{tech}
 There exists a sequence  $D_j>j$, a sequence of sections $p^j$ of $\X$ and a finite set of  sections $Q_j$ of $\X$ such that
 \begin{enumerate}[(i)]
 \item For all $q\in Q_j$ and almost every $\omega$, $d_\omega(p^j_\omega,q_\omega)<D_j(1+1/j)$.
 \item For almost every $\omega$ and all $x\in X_\omega$, there exists $q\in Q_j$ with $d_\omega(x,q_\omega)-d_\omega(x,p_\omega^j)>D_j.$
 \end{enumerate}
 \end{lem}
  \begin{proof}Let $\mathcal{F}=\{x^l\}$ be a fundamental family for $\X$. For $\omega\in\Omega$ and $j,n\in\N$, we set 
  \[D_{\omega,n}^j=\inf\left\{r>j\left|\ \exists i\in\N,\ \exists i_1,\dots,i_n,\ 
  \begin{array}{c}
  x_\omega^{i_1},\dots,x_\omega^{i_n}\in B(x_\omega^i,r(1+1/j))\\
  \mathrm{and}\ \forall l\in\N\ \exists k(l)\  d(x_\omega^l,x_\omega^{i_{k(l)}})-d(x_\omega^l,x_\omega^i)>r
  \end{array}\right.\right\}.\]
 Since $x_\omega^i$ are dense in $X_\omega$, we note that
  \[D_{\omega,n}^j=\inf\left\{r>j\left|\ \exists x\in X_\omega,
  \begin{array}{c} \exists x_1,\dots,x_n\in B(x,r(1+1/j)),\\
   \mathrm{with}\ \forall y\in X_\omega\ \exists i, \ d(x_i,x)-d(x_i,y)>r
   \end{array}
   \right.\right\}.\]
 The first way to write $D_{\omega,n}^j$ shows $\omega\mapsto D_{\omega,n}^j$ is measurable. Indeed, for $r>0$, set
   \[\Omega_{n,r}^j=\left\{\omega\in\Omega\left|\ \exists i\in\N,\ \exists i_1,\dots,i_n,\ 
  \begin{array}{l}
  x_\omega^{i_1},\dots,x_\omega^{i_n}\in B(x_\omega^i,r(1+1/j))\\
  \mathrm{and}\ \forall l\in\N\ \exists k(l)\  d(x_\omega^l,x_\omega^{i_{k(l)}})-d(x_\omega^l,x_\omega^i)>r
  \end{array}\right.\right\}.\]
   So
   \[\Omega_{n,r}^j=\bigcup_{i\in\N}\bigcup_{i_1,\dots,i_n\in\N}\bigcap_{l\in\N}\bigcup_{k\in\{i_1,\dots,i_n\}}\left\{\omega\in\Omega\left| 
   \begin{array}{l}
   \forall m\in[1,n]\ d(x_\omega^{i_m},x_\omega^i)<r(1+1/j)\\
    \mathrm{and}\ d(x_\omega^l,x_\omega^{i_{k}})-d(x_\omega^l,x_\omega^i)>r
    \end{array}
    \right.\right\}.\]
   Thus, $\Omega_{n,r}^j$ is a measurable set for $n,j\in\N$ and $r>0$. Finally, if $\omega\in\Omega_{n,r}^j$ inequalities are strict, there exists $\epsilon$ such that for all $r'$ with $|r-r'|<\epsilon$, $\omega\in\Omega_{n,r'}^j$ and
   \[\left\{\omega\in\Omega\left|\ D_{\omega,n}^j<r\right.\right\}=\bigcup_{r'<r}\Omega_{n,r'}^j=\bigcup_{r'<r,\ r'\in\mathbb{Q}}\Omega_{n,r'}^j\ .\]
  This shows $\omega\mapsto D_{\omega,n}^j$ is measurable.

  The second way to write $D_{\omega,n}^j$ shows that $\omega\mapsto D_{\omega,n}^j$ is $G$-invariant. By ergodicity, there exists $D_n^j$ such that for almost every $\omega$, $D_{\omega,n}^j=D_n^j$.\\   
   
   Fix $j$, Lemma \ref{tech} shows that for almost every $\omega$, there exists $n$ such that  $D_{\omega,n}^j<\infty$. Set 
   \[n_j(\omega)=\inf\{n\in\N|\ D_{\omega,n}^j<\infty\}.\]
   Once again, $\omega\mapsto n_j(\omega)$ is measurable and $G$-invariant. Thus there exists $n_j$ such that for almost every $\omega$,
   $n_j(\omega)=n_j$.\\
   
   We set $D_j=D_{n_j}^j+1$ and
   \[i^j(\omega)=\inf\left\{i\in\N\left|\ \exists i_1,\dots,i_n,\ 
  \begin{array}{c}
  x_\omega^{i_1},\dots,x_\omega^{i_n}\in B(x_\omega^i,D_j(1+1/j))\\
  \mathrm{et}\ \forall l\in\N\ \exists k(l)\  d(x_\omega^l,x_\omega^{i_{k(l)}})-d(x_\omega^l,x_\omega^i)>D_j-1
  \end{array}\right.\right\}.\]
  Then $\omega\mapsto i^j(\omega)$ is measurable and if we set $p^j_\omega=x_\omega^{i^j(\omega)}$ then $p^j$ is a section for all $j$. \\
  
  We endow the (countable) set of subsets with $n_j$ elements of $\N$ with an order coming from a bijection with $\N$. We set
  \[I_\omega^j=\inf\left\{i_1,\dots,i_{n_j}\left|\ \begin{array}{c}
  x_\omega^{i_1},\dots,x_\omega^{i_{n_j}}\in B(p_\omega^j,D_j(1+1/j))\\
  \mathrm{and}\ \forall l\in\N\ \exists k(l)\  d(x_\omega^l,x_\omega^{i_{k(l)}})-d(x_\omega^l,p_\omega^j)>D_j-1
  \end{array}\right.\right\}.\]
  Then $I_\omega^j$ is measurable and if we set $i_k(\omega)$ to be the $k^{\mathrm{th}}$ element of $I_\omega^j$ for $k\in[1,2,\dots,n_j]$ then $\omega\mapsto i_k(\omega)$ is measurable.\\
  
  Finally, we set
  \[Q_j=\{(x^{i_k(\omega)}_\omega)|\ k\in[1,2,\dots,n_j]\}\]
  and claimed properties hold.
  \end{proof}
\begin{proof}[Proof of Proposition \ref{mesu}] For $j\in\N$, set
 \[S_j=\left\{x\ \mathrm{ section }\ \mathrm{of}\ X\left|\ \forall f\in\mathcal{A},\ \mathring{\forall}\omega\ \exists z_\omega\in X_\omega\right.,\ 
\begin{array}{c}f_\omega(z_\omega)-f_\omega(x_\omega)=D_j-1\\
 \ \mathrm{ and }\ d_\omega(x_\omega,z_\omega)\leq(1+1/j)D_j
 \end{array}\right\}.\]
 Thanks to Lemma \ref{tech}, $S_j\neq\emptyset$ for all $j\in\N$. Then we set $C_j=(C_\omega^j)$ where
 \[C_\omega^j=\overline{Conv}\{x_\omega^j|\ x\in S_j\}.\]
 Thanks to lemma \ref{3}, $C_j$ is a subfield of $\X$. Since $S_j$ ist $\alpha$-invariant, $C_j$ is so.  Thanks to minimality of  $\X$, $C_\omega^j=X_\omega$ for almost every $\omega$. We remark that $S_j$ is stable under convex combinations and pointwise limits. Thus, for all $f\in\mathcal{A}$, almost every  $\omega$ and all $x\in X_\omega$, there exists $z\in X_\omega$ such that $f_\omega(z)-f_\omega(x)=D_j-1$ and $d_\omega(x,z)\leq(1+1/j)D_j$. We can now use Lemma \ref{1.2} which provides  $\xi_\omega\in\bo X_\omega$ such that for all $x\in X_\omega$, $-f_\omega(x)= \beta_{\xi_\omega}(x,x^0_\omega)$. We call $y^n_\omega$ the point at distance $n$ from $x_\omega^0$ such that $f_\omega(y^n_\omega)=n$ for all $n\in\N$. By construction of $f_\omega$, for all $n\in\N$, $(y_\omega^n)$ is a section of $\X$ and for almost every $\omega$, $-f=\lim_{n\to\infty}d(y^n_\omega,.)-n$. This shows that $(-f_\omega)\in\mathcal{A}$ and the same reasoning gives another $\xi'_\omega$. Moreover, if $x\in X_\omega$, the concatenation of rayons issued from $x$ towards $\xi_\omega$ and $\xi'_\omega$ is a geodesic line. Thus, $X_\omega$ is the union of geodesic line from $\xi_\omega$ to $\xi'_\omega$. Theorem II.2.14 in \cite{MR1744486} gives  the product decompostion.
\end{proof}
\begin{proof}[Proof of Proposition \ref{Rham}]
Let $X_\omega=E_\omega\times Y_\omega$ be the Euclidean de Rham decomposition of $X_\omega$ where we identify $E_\omega$ and $Y_\omega$ with to (closed convex) subsets of $X_\omega$ such that $Y_\omega\cap E_\omega=\{x_\omega\}$. Suppose that
 \begin{enumerate}[(i)]
 \item for all sections $y,z$; $\omega\mapsto d_\omega(\pi_{E_\omega}(y_\omega),\pi_{E_\omega}(z_\omega))$ is measurable.
 \end{enumerate}
 
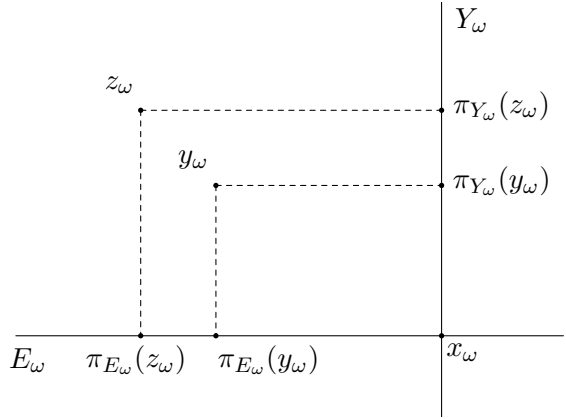
\begin{center}
 \begin{tikzpicture}[line cap=round,line join=round,>=triangle 45,x=1.0cm,y=1.0cm]
\clip(-5.8,-1.08) rectangle (1.62,4.44);
\draw [dash pattern=on 2pt off 2pt] (-4,3)-- (-4,0);
\draw [dash pattern=on 2pt off 2pt] (-3,2)-- (-3,0);
\draw [dash pattern=on 2pt off 2pt] (-3,2)-- (0,2);
\draw [dash pattern=on 2pt off 2pt] (-4,3)-- (0,3);
\draw (0,-1.08) -- (0,4.44);
\draw [domain=-5.66:1.62] plot(\x,{(-0-0*\x)/-1});
\fill [color=black] (0,0) circle (1pt);
\draw[color=black] (0.28,-0.22) node {$x_\omega$};
\fill [color=black] (-3,2) circle (1pt);
\draw[color=black] (-3.3,2.34) node {$y_\omega$};
\fill [color=black] (-4,3) circle (1pt);
\draw[color=black] (-4.28,3.34) node {$z_\omega$};
\fill [color=black] (-4,0) circle (1pt);
\draw[color=black] (-4.06,-0.32) node {$\pi_{E_\omega}(z_\omega)$};
\fill [color=black] (-3,0) circle (1pt);
\draw[color=black] (-2.3,-0.32) node {$\pi_{E_\omega}(y_\omega)$};
\fill [color=black] (0,2) circle (1pt);
\draw[color=black] (0.8,2.06) node {$\pi_{Y_\omega}(y_\omega)$};
\fill [color=black] (0,3) circle (1pt);
\draw[color=black] (0.8,3.06) node {$\pi_{Y_\omega}(z_\omega)$};
\draw[color=black] (0.4,4.2) node {$Y_\omega$};
\draw[color=black] (-5.5,-.32) node {$E_\omega$};
\end{tikzpicture}
\end{center}

Now, let  $y$ et $z$ be two sections of $\X$. We have \[d_\omega(y_\omega,\pi_{E_\omega}(y_\omega))=\sqrt{d_\omega(y_\omega,x_\omega)^2-d_\omega(x_\omega,\pi_{E_\omega}(y_\omega))^2}.\]
 Since $\pi_{E_\omega}(x_\omega)=x_\omega$,  $\omega\mapsto d_\omega(y_\omega,\pi_{E_\omega}(y_\omega))$ is measurable. Since 
 \[d_\omega(y_\omega,\pi_{E_\omega}(z_\omega))^2=d_\omega(\pi_{E_\omega}(y_\omega),\pi_{E_\omega}(z_\omega))^2+d_\omega(y_\omega,\pi_{E_\omega}(y_\omega))^2,\]
   $\omega\mapsto d_\omega(z_\omega,\pi_{E_\omega}(y_\omega))$ is measurable. This shows that for every section $z$, $(\pi_{E_\omega}(z_\omega))$ is also a section of $\X$.\\
 
 Since $d_\omega(\pi_{Y_\omega}(y_\omega),\pi_{Y_\omega}(z_\omega))^2=d_\omega(y_\omega,z_\omega)^2-d_\omega(\pi_{E_\omega}(y_\omega),\pi_{E_\omega}(z_\omega))^2$, for every section $y$ of $\X$, $(\pi_{Y_\omega}(y_\omega))$ is also a section of $\X$.\\
 
Projections on $E_\omega$ and $Y_\omega$ of sections in $\mathcal{F}$ gives fundamental families for $\mathbf{Y}=(Y_\omega)$ and $\mathbf{E}=(E_\omega)$. Thus, $\mathbf{E}$ and $\mathbf{Y}$ are subfields of $\X$. Last properties come from usual properties of the Euclidean de Rham decomposition for each $X_\omega$. \\
 
 It remains t show property (i). Fix two sections $y$ and $z$ of $\X$. For $\omega\in\Omega$, let $\mathcal{C}_\omega$ be the set of 1-Lipschitz convex functions  on $X_\omega$ which vanishes at $x^0_\omega$. For $f\in\mathcal{C}_\omega$, set $\Delta^{i,j}_\omega(f)=\frac{f(x^i_\omega)+f(x^j_\omega)}{2}-f(m^{i,j}_\omega)$ where $m^{i,j}_\omega$ is the midpoint of $[x^i_\omega,x^j_\omega]$. Since $f$ is convex, this quantity is non-negative. Let $f^k_\omega$ be the function $y\mapsto d_\omega(y,x^k_\omega)-d(x^0_\omega,x^k_\omega)$ and 
 \[K^n_\omega=\{k\in\N|\ d_\omega(x^k_\omega,x^0_\omega)\geq n\ \mathrm{et}\ \forall i,j\leq n,\ \Delta^{i,j}_\omega(f^k_\omega)\leq 1/n\}.\] We claim that following equalities hold  (with convention that the supremum of the empty subset of $\R^+$ is 0) and thus property (i) hold.
 
 \[
 d_\omega(\pi_{E_\omega}(y_\omega)),\pi_{E_\omega}(z_\omega))=\max_{f\in\mathcal{A}}|f_\omega(y_\omega)-f_\omega(z_\omega)|=\lim_{n\to+\infty}\sup_{k\in K_\omega^n}|f^k_\omega(y_\omega)-f^k_\omega(z_\omega)|.
 \]
 
 By definition of $\mathcal{A}$, we know that
 \[\max_{f\in\mathcal{A}}|f_\omega(y_\omega)-f_\omega(z_\omega)|\leq\lim_{n\to+\infty}\sup_{k\in K_\omega^n}|f^k_\omega(y_\omega)-f^k_\omega(z_\omega)|.\]
 To show the inverse inequality, set 
 \[L_\omega^n=\left\{i\in K_\omega^n\left|\ |f^i_\omega(y_\omega)-f^i_\omega(z_\omega)|\geq\sup_{k\in K_\omega^n}|f^k_\omega(y_\omega)-f^k_\omega(z_\omega)|-1/n\right.\right\},\]
let $k(n,\omega)=\min L_\omega^n$ and $g^n_\omega=f^{k(n,\omega)}_\omega$. We have
 \[\lim_{n\to+\infty}\sup_{k\in K_\omega^n}|f^k_\omega(y_\omega)-f^k_\omega(z_\omega)|=\lim_{n\to\infty}|g^n_\omega(y_\omega)-g^n_\omega(z_\omega)|.\]
 Each $\mathcal{C_\omega}$ is compact for the pointwise convergence. Thus $\prod_{\omega\in\Omega}\mathcal{C}_\omega$ is so. Up to extract a subsequence, we can suppose that for almost all $\omega$, $g^n_\omega$ converges to some $g_\omega$.Thanks to Proposition \ref{mesu}, $g=(g_\omega)$ is an element of $\mathcal{A}$. Thus, 
 \[\max_{f\in\mathcal{A}}|f_\omega(y_\omega)-f_\omega(z_\omega)|\geq\lim_{n\to+\infty}\sup_{k\in K_\omega^n}|f^k_\omega(y_\omega)-f^k_\omega(z_\omega)|\]
 and for almost all $\omega$, $g_\omega$ corresponds to a Busemann function of $E_\omega$. Thus,
 
 \[ d_\omega(\pi_{E_\omega}(y_\omega),\pi_{E_\omega}(z_\omega))\geq\max_{f\in\mathcal{A}}|f_\omega(y_\omega)-f_\omega(z_\omega)|.\]
 
Since for almost every $\omega$, $(x^i_\omega)_{i\in\N}$ is dense in $X_\omega$, for all $n$, we can find $x_\omega^k$ arbitrarily closed to the geodesic $(\pi_\omega(y_\omega), \pi_\omega(z_\omega))$ in $E_\omega$ and $k\in K_\omega^n$.
\end{proof}
\begin{cor} There are two cocycles for $G$ on $\mathbf{E}$ and $\mathbf{Y}$, $\alpha_\mathbf{E}$ and $\alpha_\mathbf{Y}$ such that $\alpha=\alpha_\mathbf{E}\times\alpha_\mathbf{Y}$.
\end{cor}
\begin{proof} For a fixed $g\in G$ and $\omega\in\Omega$, we set 
\[\mathbf{E}'=g\mathbf{E}\ \mathrm{and}\ \mathbf{Y}'=g\mathbf{Y}.\]
This gives a new de Rham decomposition $\X=\mathbf{E}'\times\mathbf{Y}'$. The second part of Proposition  \ref{Rham} shows the projection $\pi_{E_\omega}|_{E'_\omega}$ and $\pi_{Y_\omega}|_{Y'_\omega}$ are isometries. We set $\alpha_\mathbf{E}(g,\omega)=\pi_{E_{g\omega}}\circ\alpha(g,\omega)$ and  $\alpha_\mathbf{Y}(g,\omega)=\pi_{Y_{g\omega}}\circ\alpha(g,\omega)$. Thus $\alpha(g,\omega)=\alpha_\mathbf{E}(g,\omega)\times\alpha_\mathbf{Y}(g,\omega)$ for all $g$ and almost all $\omega$. Cocycle properties of  $\alpha_\mathbf{E}$ and $\alpha_\mathbf{Y}$ follow from those of $\alpha$.
\end{proof}
\subsection{Measurable Adams-Ballmann theorem}\label{finfin}
Let $G$ be locally compact second countable group and $\Omega$ an ergodic $G$-space. Let $\X$ be measurable field of CAT(0) space of finite telescopic dimension endowed with a $G$-cocycle. Let $\mathcal{F}=\{x^i\}_{i\in\N}$ be a fundamental family of $\X$.\\

 We fix a section $x^0$. For $\omega\in\Omega$, let $C_\omega$ be the locally convex linear space  of convex functions which vanishes at $x^0_\omega$ endowed with the topology of pointwise convergence. Let $\mathcal{C}_\omega$ the compact convex subspace of $C_\omega$ of functions which are moreover 1-Lipschitz. A metric on $\mathcal{C}_\omega$ can be given by the formula
 \[D_\omega(f,g)=\sum_{ i\in\N}\frac{|f(x^i_\omega)-g(x^i_\omega)|}{2^id_\omega(x_\omega^i,x^0_\omega)}.\]
 Let $\iota_\omega\colon X_\omega\to L_\omega$ defined by $\iota_\omega(x)=d_\omega(x,.)-d_\omega(x,x^0_\omega)$. We define $K_\omega$ to be the closed convex hull of $\iota_\omega(X_\omega)$. Then $\mathbf{K}=(K_\omega)_{\omega}$ is a measurable field of compact spaces. A fundamental family is given by rational convex combinations of elements in $\{\iota_\omega(x^i_\omega)\}_{i\in\N}$. A section $f=(f_\omega)$ of $\mathbf{K}$ is called \emph{affine}  if for almost all $\omega$, $f_\omega$ is an affine function on $X_\omega$.
\begin{prop}\label{ouf}If $\X$ has no Euclidean factor, is minimal and is not reduced to a section then $\mathbf{K}$ does not contain any affine section. 
\end{prop}
\begin{proof} We use same notations as in the proof of Proposition \ref{Rham}. We introduce the following function $\tau\colon\Omega\to\R$ where
\[\tau(\omega)=\inf_{f\in K_\omega}\sup_{i,j}\Delta_\omega^{i,j}(f).\]
The function $\tau$ is measurable and $G$-invariant thus, by ergodicity, it is an essentially constant function. Since $K_\omega$ is compact $\tau(\omega)=0$ means exactly that $K_\omega$ contains an affine function. The same way we define $\upsilon\colon\Omega\to\R$ where
   \[\upsilon(\omega)=\inf_{f\in i_\omega(X_\omega)}\sup_{i,j}\Delta_\omega^{i,j}(f).\]
 Then $\upsilon$ is also an essentially constant function. We remark that for almost every $\omega$, $\upsilon(\omega)=\inf_{f\in\mathcal{A}}\sup_{i,j}\Delta_\omega^{i,j}(f)$. Since $\X$ has no Euclidean factor and is minimal, Proposition \ref{mesu} implies that $\mathcal{A}$ is empty. Thus, there is $\varepsilon>0$ such that for almost every $\varepsilon$, $\upsilon(\omega)=\varepsilon$. Lemma 4.10 in \cite{MR2558883} shows that $\tau(\omega)>0$ almost everywhere and then $\mathbf{K}$ can not have any affine section.
\end{proof}
We define $\mathcal{C}(\mathbf{K})$ to be the measurable field of Banach spaces $(\mathcal{C}(K_\omega))$ where  $\mathcal{C}(K_\omega)$ is the Banach space of continuous functions on the compact metrizable space $K_\omega$. Each $\mathcal{C}(K_\omega)$ is endowed with the $\sup$-nom. For details about the measurable structure, we refer to \cite{Anderegg} or \cite{henry}.\\

If $\beta(g,\omega)(f)=f\circ\alpha(g^{-1},g\omega)-f($ for $f\in K_\omega$ then we define a cocycle $\gamma$ via the formula
\[\gamma(g,\omega)\varphi(f)=\varphi(\beta(g^{-1},g\omega)(f))\]
for $f\in K_{g\omega}$ and $\varphi$ is a section of $\mathcal{C}(K_\omega)$. We endow the dual field $\mathcal{C}(\mathbf{K})^*$ with the dual cocycle $\gamma^*$.\\

Let $\mathcal{M}(\mathbf{K})$ be the compact convex subfield of $\mathcal{C}(\mathbf{K})^*_1$ (the measurable field of unit balls in duals). It is invariant under $\gamma^*$ and moreover if $\mu\in\mathcal{M}(K_\omega)$ then
\[\gamma(g,\omega)^*\mu=\beta(g,\omega)_*\mu.\]
We recall (see \cite[4.1.4]{MR776417} and reference therein) there exists a continuous map, called \emph{barycenter}, $b\colon\mathcal{M}(K)\to K$ for $K$ a compact convex subset of a locally convex space. This map is defined on convex combinations of Dirac masses by
\[b(\sum\lambda_i\delta_{k_i})=\sum\lambda_ik_i.\]
The density of such combinations (for the  weak-* topology) permits to define $b$ everywhere. Moreover, if $T\colon K\to K'$ is an affine map between two compact convex subsets of some locally convex spaces and $b,b'$ are the respective barycenter maps then $T\circ b(\mu)=b'(T_*\mu)$ for every $\mu\in\mathcal{M}(\mathbf{K})$.\\

So if $\mu$ is a section of $\mathcal{M}(\mathbf{K})$ invariant for the cocycle $\gamma^*$ then we define $k_\omega=b_\omega(\mu_\omega)$ where $b_\omega$ is the barycenter map $b_\omega:\mathcal{M}(K_\omega)\to K_\omega$. Since for all $g$ and almost every $\omega$ $\beta(g,\omega)$ is affine, we have
 \[(\beta\cdot k)_\omega=\beta(g,g^{-1}\omega)k_{g^{-1}\omega}=b_{\omega}(\beta(g,g^{-1}\omega)_*\mu_{g^{-1}\omega})=\beta_\omega(\mu_\omega)=k_\omega .\]
This gives an invariant section of $\mathbf{K}$.
\begin{proof}[Proof of Theorem \ref{AB}] We suppose there is no invariant section of  $\bo\X$. Thanks to Proposition \ref{ibs}, there exists a minimal invariant subfield $\X'$ of $\X$. Let $\X'=\mathbf{E}\times\mathbf{Y}$ be a Euclidean de Rham decomposition of $\X'$. The minimality of $\X'$ implies the minimality of $\mathbf{E}$ and $\mathbf{Y}$ under cocycles $\alpha_\mathbf{E}$ and $\alpha_\mathbf{Y}$. Then it suffices to show $\mathbf{Y}$ is reduced to a section.\\

Fix a section $x$ of $\mathbf{Y}$. Let $\mathbf{K}$ be the measurable field of compact spaces previously introduced at the beginning of section \ref{finfin}, relatively to $\mathbf{Y}$. By amenability of the action $G\action\Omega$, $\mathcal{M}(\mathbf{K})$ has an invariant section. The previous discussion shows $\mathbf{K}$ has also an invariant section.  Let $f$ be an invariant section of $\mathbf{K}$, this means for all section $y$ of $\mathbf{Y}$ and almost every $\omega$,
 \begin{equation}\label{cococycle}
 f_\omega(y_\omega)=f_{g^{-1}\omega}(\alpha_\mathbf{Y}(g^{-1},\omega)y_{\omega})-f_{g^{-1}\omega}(\alpha_\mathbf{Y}(g^{-1},\omega)x_{\omega}).
 \end{equation}
 This means the quantity $f_{g^{-1}\omega}(\alpha_\mathbf{Y}(g^{-1},\omega)y_{\omega})-f_\omega(y_\omega)$ does not depend on $y$. So, we set $c(g,\omega)=f_{g\omega}(y_{g\omega})-f_{\omega}(\alpha_\mathbf{Y}(g^{-1},g\omega)y_{g\omega})$ and  then $c:G\times \Omega\to\R$ s an additive  cocycle. We introduce the three measurable subsets of $\Omega$,
 \begin{eqnarray*}
 \Omega_{\min}&=&\left\{\omega\in\Omega|\ f_\omega\ \text{has a minimum}\right\}, \\
 \Omega_{\inf}&=&\left\{\omega\in\Omega|\ f_\omega\ \text{does not have a minimum and}\ \inf f_\omega>-\infty\right\},\\
 \Omega_{-\infty}&=&\left\{\omega\in\Omega|\ \inf f_\omega=-\infty\right\}.
 \end{eqnarray*}
The equation (\ref{cococycle}) can be written  by replacing $\omega$ by $g\omega$,
 \[f_{g\omega}=f_\omega \circ\alpha_\mathbf{Y}(g^{-1},g\omega)+c(g,\omega).\]
 Thus the three previous subsets are $G$-invariant. Their union has full measure and since $G\curvearrowright\Omega$ is ergodic, one of them has full measure.\\
 
 If $\Omega_{\min}$ has full measure then we define $Y'_\omega=f^{-1}_\omega(\min f_\omega)$ and this gives a $G$-invariant subfield $\mathbf{Y}'$ of $\mathbf{Y}$. By minimality $\mathbf{Y}'=\mathbf{Y}$. This shows that for almost every $\omega$, $f_\omega$ is a constant function on  $Y_\omega$. But since $\mathbf{Y}$ has no Euclidean factor and is minimal, Proposition \ref{ouf} shows that $\mathbf{Y}$ is reduced to a section. \\
 
 If  $\Omega_{-\infty}$ has full measure, we define $Y^r_\omega=f_\omega^{-1}(]-\infty,r])$ for $r\in\R^-$ and if $\Omega_{\inf}$ has full measure, we define for $r\in\R^+$, $Y^r_\omega=f_\omega^{-1}(]\inf f_\omega,\inf f_\omega+r])$. In the two cases, $\mathbf{Y}^r$ is subfield wich satisfies 
 \[ \alpha_\mathbf{Y}(g,\omega)Y^r_\omega=Y_{g\omega}^{r+c(g,\omega)}\] for all $g$, almost every $\omega$ and all $r>0$. We choose a countable dense subset $D$ of rational numbers in $\R^+$ or $\R^-$. Thanks to Proposition \ref{bsec} (used for $(\mathbf{Y}^r)_{r\in D}$) we construct a section of $\bo\mathbf{Y}$ and thanks to Proposition \ref{independant} it is $\alpha_\mathbf{Y}$-invariant and this gives also an $\alpha$-invariant section of $\bo\X$. 
\end{proof}
\section{Furstenberg maps}\label{sfm}
Let us start by recalling the following lemma.
\begin{lem}[{\cite[Proposition 1.8.(ii)]{MR2558883}}]\label{19.2}Let $G$ be a group acting on $X_p(\K)$ by isometries without fixed at infinity. Then there exists a minimal (non-empty) $G$-invariant closed convex subspace $X$ of $X_p(\K)$.
\end{lem}

Now, we suppose  we are under the hypotheses of Theorem \ref{fm}. That is $G$ is a locally compact second countable group which acts measurably and in a non-elementary way by isometries on a space $X_p(\K)$. The action $G\action X_p(\K)$ is \emph{non-elementary} if there no $G$-invariant Euclidean subspace of $X_p(\K)$ nor fixed point at infinity.\\

Let $B$ be a $G$-boundary and let $Y$ be a complete CAT(0) space on which $G$ acts continuously by isometries. We will consider the constant field $\mathbf{Y}$ over $B$ such that for all $b\in B$, $Y_b=Y$. We choose a dense countable family $\mathcal{D}$ of points in $Y$. A fundamental family of $\mathbf{Y}$  is given by elements $(x_b)$ such that all $x_b$ are equal to a same element of $\mathcal{D}$. The group $G$ acts on $\mathbf{Y}$ via a cocycle $\alpha$ where $\alpha(g,b)=g$ for every $g$ and $b$. With this definition, we remark that an invariant section $x=(x_b)$ coincides with a measurable $G$-equivariant map from $B$ to $Y$,  $b\mapsto x(b)=x_b$. Indeed, $x$ is an invariant section means that $\alpha(g,b)x_b=x_{gb}$ for all $g$ and almost every $b$ and $b\mapsto x(b)$ is $G$-invariant means that $x(gb)=gx(b)$ for almost every $b$.This is just a matter of vocabulary.\\

We will set $\X$ to be the constant field such that for all $b$, $X_b=X_p(\K)$. Let $Y$ be a closed convex subset of $X$ endowed with a continuous $G$-action by isometries. We consider the constant field $\mathbf{Y}$ with the cocycle $\alpha_Y(g,\omega)=g$ for all $g$ and almost every $\omega$. Since  $Y\subset X_p(\K)$ then each $Y_b$ can be seen as a non-empty closed convex subset  of $X_p(\K)$.

\begin{prop}\label{ies}If there is an invariant Euclidean subfield $\mathbf{E}$ of $\mathbf{Y}$ then for almost every $(b,b')$, $\bo E_b\cap \bo E_{b'}\neq\emptyset$ or there is a Euclidean subspace $E_0$ of $Y$ such that for almost all $b$, $E_b=E_0$.
\end{prop}
\begin{proof}Inspired by Gromov-Hausdorff convergence,  we introduce the function $d_{GH}$ of Euclidean subspaces of $X_p(\K)$. We emphasize that this function is \emph{not} a distance and does not characterizes the Gromov-Hausdorff convergence. If $E,F$ are Euclidean subspaces of $Y$ we define $d_{GH}(E,F)=$
\[\inf\left\{1/r\left|\exists x\in E,\ y\in F,\  d_H\left(E\cap\overline{B}(x,r),F\cap\overline{B}(y,r)\right)\leq 1/r\right.\right\}
\]
where $d_H$ is the Hausdorff distance on closed bounded subspaces of $X_p(\K)$ .  The action of $\Isom(X_p)$ on Euclideans subspaces  preserve $d_{GH}$. Moreover, for any $r>0$, we can find countably many subsets  
\[U(E_i,x_i,r)=\left\{F\mathrm{\ Euclidean\ subspace }\left|\ d_H(F\cap B(x_i,r),E_i\cap B(x_i,r))<1/r\right.\right\}\]
 which cover the set of all Euclidean subspaces of $X_p(\K)$. Actually, $\O_{p,\infty}^F(\K)$ acts transitively on the pairs $(x,A)$ where $A$ is maximal Euclidean subspace and $x$ a point in $A$. Moreover if $g\to e$ in $\O_{p,\infty}^F(\K)$ (with the norm-topology) then for all $x\in X_p(\K)$ and $r>0$ the restriction of $g$ to $\overline{B}(x,r)$ converges uniformly to  the identity map. Fix a point $x$ in a maximal Euclidean subspace $A$. For each dimension $0\leq i\leq p$, choose a countable family $(A^i_n)_{n\in\N}$ of Euclidean subspaces in $A$ such that for all Euclidean subspaces $E\subset A$ containing $x$ and all $\varepsilon>0$, there exist $i,n$ such that $d_{GH}(E,A^i_n)<\varepsilon$. Now, choose a dense subset $\{g_j\}_j$ of $\O_{p,\infty}^F(\K)$ then  the countable family $\{(g_jx,g_jA_n^i)\}_{(i,j,n)}$ induces a covering as above. \\
 
 Let $E,F$ be two subspaces such that $d_{GH}(E,F)=0$. We claim that $E=F$ or $\bo E \cap\bo F\neq\emptyset$.
 Indeed, there exist  sequences $(x_n),(y_n)$ such that for $n\in \N$,  \[d_H\left(E\cap\overline{B}(x_n,n),F\cap\overline{B}(y_n,n)\right)\leq 1/n\ .\] We can moreover suppose that $x_n\in E$ for all $n$. Since $\overline{E}$ is compact, we can also suppose this sequence is convergent. If the limit is in $E$ then $F=E$ and if the limit is in $\bo E$ then it is also the limit of $(y_n)$ and thus a point in $\bo E\cap\bo F$.\\
 
 Since the Euclidean field is invariant $E_{gb}=gE_b$ and then $b,b'\mapsto d_{GH}(E_b,E_{b'})$ in invariant. Thanks to double ergodicity, this function is essentially equal to some $r\geq0$. If $r>0$, we choose a Euclidean subspace $E$ and a point $x$ such that $U(E,x,r/2)$ has positive measure for the image measure by $b\mapsto E_b$ and we find $P\subset\Omega\times\Omega$ which has positive measure such that $(b,b')\in P$, $d_{GH}(E_b,E_{b'})<r$. Thus $r=0$.\\
 
 This implies that for almost all $(b,b')$, $E_b=E_{b'}$ or $\bo E_b\cap\bo E_{b'}\neq\emptyset$. The set $\{(b,b')|\ E_b=E_{b'}\}$ is $G$-invariant and measurable then by double ergodicity it is null or conull. In the first case this implies that that for almost all $(b,b')$ $\bo E_b\cap\bo E_{b'}\neq\emptyset$ and in the second one this implies by Fubini's theorem that there exists $b$ such that for almost every $b'$, $E_b=E_{b'}$.
\end{proof}
\begin{proof}[Proof of Theorem \ref{fm}] Since we suppose the action $G\action X_p(\K)$ is non-elementary then Lemma \ref{19.2} implies there exists a minimal closed convex invariant subspace $X$ of $X_p(\K)$. Let $X=E\times Y$ be the Euclidean de Rham decomposition of $X$. It is a property of Euclidean de Rham decomposition that the action $G\action X$ is diagonal and thus it induces an action  $G\action Y$ which is also minimal. Moreover, fixing a point $x\in X$, $Y$ identifies with a unique totally geodesic subspace of $X$ which contains $x$. We note this subspace is not invariant under the action of $G$ on $X$. However the embedding $\bo Y\subset \bo X$ does not depend on the identification. Thus $\bo Y$ is $G$-invariant (as subspace of $\bo X$) and the extension on $\bo Y$ of the actions $G\action X$ and $G\action Y$ are the same. We will retain the following data on $Y$

\begin{itemize}
\item[$\bullet$] The space $Y$ is a complete CAT(0) space of finite telescopic dimension.
\item[$\bullet$] The action $G\action Y$ is minimal.
\item[$\bullet$] The boundary $\bo Y$ is a subset of $\bo X_p(\K)$.
\end{itemize}

Let $\mathbf{Y}$ be the constant field over $B$ associated  to $Y$.  We apply Theorem \ref{AB} to this field. This gives an invariant section of $\bo\mathbf{Y}$ or an invariant Euclidean subfield. An invariant section of $\bo\mathbf{Y}$ gives a $G$-equivariant map $B\to\bo Y\subset\bo X_p(\K)$ when the latter is equipped with the visual topology and is thus homeomorphic to a separable Hilbert sphere. \\

Suppose now $\mathbf{E}$ is an invariant Euclidean subfield of $\mathbf{Y}$. Since $\mathbf{Y}$ is a subfield of $\X$, $\mathbf{E}$ is also a subfield of $\X$ and we can apply proposition and then for almost every $(b,b')$, $\bo E_b\cap\bo E_{b'}\neq\emptyset$ or there is $E_0$ Euclidean subspace of $Y$ which is invariant. In the second case this means that $E\times E_0$ is a Euclidean invariant subspace of $X_p$. So we suppose that  for almost every $(b,b')$, $\bo E_b\cap\bo E_{b'}\neq\emptyset$.\\

We consider now the building structure $\mathcal{I}_p(\K)$ on $\bo X_p(\K)$. Since $\bo E_b$ is a subspace of an apartment of $\mathcal{I}_p(\K)$ we note $C_b$ the minimal subcomplex  of $\mathcal{I}_p(\K)$ which contains $\bo E_b$. This is a finite subcomplex.\\

If $C$ is a subcomplex included in an apartment of $\mathcal{I}_p(\K)$ we call the \emph{type} of $C$, its class under the action of Isom$(X_p(\K))$. The set of types is finite because each apartment is a finite subcomplex and Isom$(X_p(\K))$ acts transitively on apartments of  $\mathcal{I}_p(\K)$. We know that for almost every $(b,b')$, $C_b\cap C_{b'}$ is a non-empty subcomplex. Since the type is invariant under the action of $G$, double ergodicity implies that there is a type $D$ such that for almost all $(b,b')$ $C_b\cap C_{b'}$ has type $D$.\\

We claim that the map $(b,b')\mapsto C_b\cap C_{b'}$ is essentially constant. We set $D_{b,b'}=C_b\cap C_{b'}$. Fubini's theorem implies that there is a conull measurable set $B_0\subseteq B$ such that for all $b\in B_0$ there is $B_b\subseteq B$ which is conull such that for all $b'\in B'$, $D_{b,b'}$ has type $D$. We fix $b_1,b_2\in B_0$ and we set $C_1=C_{b_1}$ and $C_2=C_{b_2}$. Since there is a finite number of subcomplexes of $C_1$ there exists $B_1\subset B$ with positive measure and a subcomplex $D_1$ of type $D$ such that for all $b\in B_1$
\[C_1\cap C_b=D_1.\]
Then for all for $b,b'\in B_1$, $D_{b,b'}=D_1$. Since there is also a finite number of subcomplexes of $C_2$ there is a measurable subset of positive measure $B_2\subset B_1$ such that for all $b,b'\in B_2$, $D_{b,b'}=C_2\cap C_b=C_2\cap C_{b'}$. Since $C_1\cap C_2$ has type $D$ we have $C_1\cap C_2=D_1$. This implies that $b\mapsto D_{b_1,b}$ is essentially equal to $D_1$. Since for almost $b,b'$ $D_{b_1,b}=D_1=D_{b_1,b'}$ and $D_{b,b'}$ has type $D$, we have $D_{b,b'}=D_1$. This shows this complex is $G$-invariant.\\

Let $f$ be the mean of the Busemann functions (of $Y$) associated with the circumcenters of the cells of $D_1$ and that vanish at some point $x_0\in Y$. Since $G$ permutes the circumcenters of the cells of $D_1$, the function $f$ is quasi-invariant.  That is for all $g\in G$ and $x\in Y$, $f(gx)=f(x)+f(gx_0)$. Thus $g\mapsto f(gx_0)$ is a homomorphism. If $f$ has a minimum then this homomorphism is trivial and the subset $Z$ where $f$ achieves its minimum is a closed convex non-empty closed subspace of $Y$ which $G$-invariant. By minimality, $Z=Y$ and then $f$ is constant on $Y$. This means that $Y$ has at least one affine Busemann. This is a contradiction because $Y$ has a trivial de Rham factor.\\

If $f$ does not achieve its minimum the center of directions associated with the nested family of sublevel sets is a fixed point at infinity for the action $G\action Y$ and thus there is a fixed point at infinity for the initial action $G\action X_p(\K)$. This is a contradiction with non-elementarity.\\

Finally, we a have obtained a measurable $G$-equivariant map $\varphi\colon B\to \bo X_p(\K)$. 

\end{proof}
\bibliographystyle{../../Latex/Biblio/halpha}
\bibliography{../../Latex/Biblio/biblio.bib}
\end{document}